\documentclass[12pt]{amsart}
\usepackage{amsmath}
\usepackage{amsthm}
\usepackage{verbatim}
\usepackage{amssymb}
\usepackage[colorlinks=true,pagebackref]{hyperref}
\usepackage[dvipsnames,usenames]{color}
\hypersetup{linkcolor=RawSienna,anchorcolor=BurntOrange,citecolor=OliveGreen,filecolor=BlueViolet,menucolor=Yellow,urlcolor=OliveGreen}

\newtheorem{lemma}{Lemma}[section]
\newtheorem{theorem}[lemma]{Theorem}

\newtheorem{proposition}[lemma]{Proposition}

\theoremstyle{definition}
\newtheorem{definition}[lemma]{Definition}
\newtheorem{remark}[lemma]{Remark}

\newtheorem{notation}[lemma]{Notation}
\newtheorem{example}[lemma]{Example}

\theoremstyle{remark}
\newtheorem*{remark*}{Remark}
\newtheorem*{note*}{Note}

\newcommand{\Frac}{\operatorname{Frac}}

\newcommand{\Spec}{\operatorname{Spec}}

\newcommand{\Hom}{\operatorname{Hom}}

\newcommand{\ord}{\operatorname{ord}}
\newcommand{\val}{\operatorname{val}}
\newcommand{\trdeg}{\operatorname{tr.deg}}

\newcommand{\Cont}{\operatorname{Cont}}
\newcommand{\divi}{\operatorname{div}}

\newcommand{\x}{\overline{x}}

\newcommand{\CC}{\mathbf{k}}
\newcommand{\ZZ}{\mathbb{Z}}

\newcommand{\ZZZ}{\mathbb{Z}_{\geq 0} \cup \{\infty \}}

\newcommand{\CCC}{\mathbb{C}}

\begin{document}
\title{Zero dimensional arc valuations on smooth varieties}
\author{Yogesh More}
\address{Department of Mathematics, University of Michigan,
Ann Arbor, MI 48109, USA}
\email{{\tt yogeshmore80@gmail.com}}

%\date{January 20, 2008}
\begin{abstract}
Let  $X$ be a nonsingular variety (with $\dim X \geq 2$) over an algebraically closed field $\CC$ of characteristic zero. Let $\alpha:\Spec \CC[[t]] \to X$ be an arc on $X$, and let $v=\ord_\alpha$ be the valuation given by the order of vanishing along $\alpha$. We describe the maximal irreducible subset $C(v)$ of the arc space of $X$ such that $\val_{C(v)}=v$. We describe $C(v)$ both algebraically, in terms of the sequence of valuation ideals of $v$, and geometrically, in terms of the sequence of infinitely near points associated to $v$. When $X$ is a surface, our construction also applies to any divisorial valuation $v$, and in this case $C(v)$  coincides with the one introduced in \cite[Example 2.5]{mustata}.
\end{abstract}

\thanks{The author is grateful to have been supported by the NSF RTG grant DMS-0502170.}

\maketitle

\section{Introduction}\label{intro}

Let $X$ be a variety over a field $\CC$. A $\CC$-arc $\gamma$ on $X$ is a morphism of $\CC$-schemes $\gamma:\Spec \CC[[t]] \to X$. There is a scheme $X_\infty$, called the arc space of $X$, which parametrizes the arcs on $X$. We refer the reader to \cite[Section 2]{jet-bible} for the construction of $X_\infty$. Denote the closed point of $\Spec \CC[[t]]$ by $o$.

In this paper, I study valuations $\ord_\gamma:\mathcal{O}_{X,\gamma(o)} \to \ZZZ$ given by the order of vanishing along a $\CC$-arc $\gamma:\Spec \CC[[t]] \to X$.  Such valuations are precisely the $\ZZZ$-valued valuations with transcendence degree zero. I associate to $\ord_\gamma$ several different natural subsets of the arc space $X_\infty$. I prove if $\gamma$ is a nonsingular arc, then these subsets associated to $\ord_\gamma$ are equal. Furthermore, I show this subset is irreducible, and the valuation given by the order of vanishing along a general arc of this subset is equal to the original valuation $\ord_\gamma$.  

The motivation for this project was the discovery by Ein, Lazarsfeld, and Musta\c{t}\v{a} \cite[Thm. C]{mustata} that divisorial valuations (equivalently, valuations with transcendence degree $\dim X-1$)  correspond to a special class of subsets of the arc space called cylinders. More specifically, for a divisorial valuation $\val_E$ given by the order of vanishing along a prime divisor $E$ over $X$, there is an irreducible cylinder $C_{\divi}(E) \subseteq X_\infty$ such that for a general arc $\gamma \in C_{\divi}(E)$, we have that the order of vanishing of any rational function $f \in \CCC(X)$ along $\gamma$ equals its order of vanishing along $E$. In symbols, $\ord_\gamma(f) = \val_E(f)$ for all $f \in \CCC(X)$. Conversely, it is shown in \cite[Thm. C]{mustata} that every valuation given by the order of vanishing along a general arc of a cylinder is a divisorial valuation. 

The goal of this paper is to investigate whether other types of valuations, besides divisorial ones, have a similar interpretation via the arc space. We find there is a nice answer for valuations given by the order of vanishing along a nonsingular arc on a nonsingular variety $X$. If $X$ is a surface, all valuations with value group $\mathbb{Z}^r$ (lexicographically ordered) for some $r$ are equivalent to a valuation of this type. One can interpret our results as being complementary to those of Ein et. al. as follows. Both say that valuations are encoded in a natural way as closed subsets of the arc space. We address the case when the transcendence degree is zero, whereas Ein et. al. study the case of valuations with transcendence degree $\dim X-1$.

\subsection{Valuations and subsets of the arc space}

In this section, I begin by explaining the relationship between valuations on a variety $X$ over a field $\CC$ and subsets of the arc space $X_\infty$ of $X$. I then construct several natural subsets of the arc space that one might associate to a valuation. One of the main results of this paper is that for a large class of valuations, these different constructions agree, i.e. they define the same subset of the arc space.

We need to introduce some notation. An arc $\gamma:\Spec \CC[[t]] \to X$ gives a $\CC$-algebra homomorphism $\gamma^*:\widehat{\mathcal{O}}_{X, \gamma(o)} \to \CC[[t]]$, where $o$ denotes the closed point of $\Spec \CC[[t]]$. We define a valuation $\ord_\gamma: \widehat{\mathcal{O}}_{X,\gamma(o)} \to \ZZZ$ by $\ord_\gamma(f)=\ord_t \gamma^*(f)$ for $f \in \widehat{\mathcal{O}}_{X,\gamma(o)}$. If $\gamma^*(f)=0$, we will adopt the convention that $\ord_\gamma(f) = \infty$. 

Given an ideal sheaf $\mathfrak{a} \subseteq \mathcal{O}_X$ on X we set $\ord_{\gamma}(\mathfrak{a})=\displaystyle \min_{f\in \mathfrak{a}_{\gamma(o)}} \ord_\gamma(f)$. For a nonnegative integer $q$, we define the $q$-th order \textit{contact locus} of $\mathfrak{a}$ by \begin{equation}
\Cont^{\geq q}(\mathfrak{a}) = \{\gamma: \Spec \CC[[t]] \to X \mid \ord_{\gamma}(\mathfrak{a})\geq q \}.
\end{equation}

The following definition appears in \cite[p.3]{mustata}, and provided, at least for us, the initial link between valuations and arc spaces:

\begin{definition}\label{intro-valc}
Let $C \subseteq X_\infty$ be an irreducible subset. Assume $C$ is a cylinder (\cite[p.4]{mustata}). Define a valuation $\val_C:\CC(X) \to \mathbb{Z}$ on the function field $\CC(X)$ of $X$ as follows. For $f \in \CC(X)$, set $$\val_C(f)=\ord_\gamma(f)$$ for general $\gamma \in C$. Equivalently, if $\alpha \in C$ is the generic point of $C$, then $\val_C(f)=\ord_\alpha(f)$. (Caveat: $\alpha$ need not be a $\CC$-valued point of $X_\infty$. See Remark \ref{point}.)
\end{definition}
 
It turns out that the condition that $C$ is a cylinder implies that $\val_C(f)$ is always finite. If we drop the assumption that $C$ is a cylinder, then the map $\ord_\alpha$ (where $\alpha$ is the generic point of $C$) is a $\ZZZ$-valued valuation on $\mathcal{O}_{X, \alpha(o)}$.  

We now describe a way to go from valuations centered on $X$ to subsets of the arc space. Following Ishii \cite[Definition 2.8]{ishii3}, we associate to a valuation $v$ a subset $C(v) \subseteq X_\infty$ in the following way.

\begin{definition}\label{c-v}
Let $p \in X$ be a (not necessarily closed) point. Let $v:\widehat{\mathcal{O}}_{X, p} \to \ZZZ$ be a valuation. Define the \textit{maximal arc set} $C(v)$ by $$C(v)=\overline{\{\gamma \in X_\infty \mid \ord_\gamma = v, \ \gamma(o)=p\}} \subseteq X_\infty,$$ where the bar denotes closure in $X_\infty$. 
\end{definition}

If we start with an irreducible subset $C$, we get a valuation $\val_C$ by Definition \ref{intro-valc}. We can then form the subset $C(\val_C)$ as in Definition \ref{c-v}. We have $C \subseteq C(\val_C)$ because $C(\val_C)$ contains the generic point of $C$. In general, we do not have equality.

We can associate another subset of $X_\infty$ to a valuation $v$ on a nonsingular variety $X$ as follows. Let $\{E_q\}_{q \geq 1}$ be the sequence of divisors formed by blowing up successive centers of $v$ (see Definition \ref{divisor-limit}). Following \cite[Example 2.5]{mustata}, to each divisor $E_q$ we associate a cylinder $C_q=C_{\divi}(E_q) \subseteq X_\infty$. Using notation we will explain in Chapter \ref{main}, we will define $C_q=\mu_{q\infty}(\Cont^{\geq 1}(E_q))$. In words, $C_q$ is simply the set of arcs on $X$ whose lift to $X_{q-1}$ (a model of $X$ formed by blowing up $q-1$ successive centers of $v$) has the same center on $X_{q-1}$ as $v$. This collection $\{C_q\}_{q \geq 1}$ of cylinders forms a decreasing nested sequence. We take their interesection, $\bigcap_{q \geq 1}C_q$, to get another subset of $X_\infty$ that is reasonable to associate with $v$.

On the other hand, another way the valuation $v$ can be studied is through its valuation ideals $\mathfrak{a}_q=\{f \in \widehat{\mathcal{O}}_{X,p} \mid v(f) \geq q\}$, where $q$ ranges over the positive integers. The set $\bigcap_{q \geq 1}\Cont^{\geq q}(\mathfrak{a}_q)$ is yet another reasonable set to associate with $v$.

Given an arc $\alpha:\Spec \CC[[t]] \to X$, we have a valuation $v=\ord_\alpha: \mathcal{O}_{X, \alpha(o)} \to \ZZZ$. As described earlier,  $v$ induces a discrete valuation on the function field of the subvariety $Y=\overline{\alpha(\eta)}$, where $\eta$ is the generic point of $\Spec \CC[[t]]$. We have a closed immersion $Y_\infty \subseteq X_\infty$. Denote by $\pi_X:X_\infty \to X$ the canonical morphism that sends an arc $\gamma \in X_\infty$ to its center $\gamma(o) \in X$.
We associate to $\ord_\alpha$ the set 
\begin{equation}\label{set-1}
\pi_X^{-1}(\alpha(o)) \cap Y_\infty \subseteq X_\infty.
\end{equation}

Let $R=\{a \circ h \in X_\infty \mid h:\Spec \CC[[t]] \to \Spec \CC[[t]]\}$. In words, $R$ is the set of $\CC$-arcs that are reparametrizations of $\alpha$. 

The main result of this paper is that for valuations $v=\ord_\alpha$, all five of these closed subsets ($C(v)$, $\bigcap_{q \geq 1}C_q$, $\bigcap_{q \geq 1}\Cont^{\geq q}(\mathfrak{a}_q)$, $\pi_X^{-1}(\alpha(o)) \cap Y_\infty$, $R$) are equal. Furthermore, this subset is irreducible, and the valuation given by the order of vanishing along a general arc of this subset is equal to $v$. 

For convenience, we will assume the arc $\alpha$ we begin with is normalized, that is, the set $\{v(f) \mid f \in \widehat{\mathcal{O}}_{X,p}, \ 0<v(f)<\infty\}$ (where $v=\ord_\alpha$) is non-empty and the greatest common factor of its elements is $1$. Every arc valuation taking some value strictly between $0$ and $\infty$ is a scalar multiple of a normalized valuation.

Also, we restrict ourselves to considering the $\CC$-arcs in the sets described above. We denote by $(X_\infty)_0$ the subset of points of $X_\infty$ with residue field equal to $\CC$. If $D \subseteq X_\infty$, then we set $D_0 = D \cap (X_\infty)_0$.
 
\begin{theorem}\label{intro-main-thm}
Let $\alpha:\Spec \CC[[t]] \to X$ be a normalized arc on a nonsingular variety $X$ ($\dim X \geq 2$) over an algebraically closed field $\CC$ of characteristic zero. Set $v=\ord_\alpha$. Then the following closed subsets of $(X_\infty)_0$ are equal:

$$C(v)_0=(\bigcap_{q \geq 1}C_q)_0 = (\bigcap_{q \geq 1} \Cont^{\geq q}(\mathfrak{a}_q))_0=(\pi_X^{-1}(\alpha(o)) \cap Y_\infty)_0=R.$$
Furthermore, the valuation given by the order of vanishing along a general arc of this subset is equal to $v$.

\end{theorem}

\

\

\begin{remark}
If $X$ is a surface and if $v$ is a divisorial valuation, then $\bigcap_{q>0} C_q$ equals the cylinder $C_r$ associated to $v$ in \cite[Example 2.5]{mustata}, where $r$ is such that $p_r$ is a divisor. 
\end{remark}

\subsection{Outline of the paper}
In Section \ref{background} we recall some basic terminology and results regarding arc spaces. In Section \ref{arcs-valns} we define {arc valuations}, and we compare them with other notions of a valuation. In Section \ref{zero-valns} we show that $\CC$-arc valuations can be desingularized. We will need this result in Section \ref{main}, where we study  $CC$-arc valuations on nonsingular varieties. We first study the case of a nonsingular arc valuation. Later we consider more general arc valuations and prove Theorem \ref{intro-main-thm}. 

\subsection{Acknowledgements}
This paper is part of the author's Ph.D. thesis, and he thanks his thesis advisor Karen E. Smith for suggesting the problem of valuations in arc spaces and for many helpful discussions throughout the course of working on this problem. The author is also indebted to Mircea Musta\c{t}\v{a} for many mathematical ideas, as well as for correcting an earlier draft of this paper. The author is grateful to Mel Hochster for providing an alternative characterization of normalized arcs (Prop. \ref{hochster-genius}) and a key lemma (Lemma \ref{key-lemma}). Conversations with Mattias Jonsson were also useful.

\section{Background on Arc spaces}\label{background}

Let $X$ be a variety over a field $\CC$. Let $\CC \subseteq K$ be a field extension. The arc space $X_{\infty}$ is a scheme over $\CC$ whose $K$-valued points are morphisms $\Spec K[[t]] \rightarrow X$ of $\CC$-schemes, since we have 
\begin{equation}\label{arc-functor}
\Hom(\Spec K, X_\infty)  = \Hom(\Spec K[[t]], X).
\end{equation}
In particular, when $X$ is affine, giving a $K$-valued point of $X_\infty$ is the same thing as giving a homomorphism of $\CC$-algebras $\Gamma(X, \mathcal{O}_X) \to K[[t]]$. 

\begin{definition}\label{defn-of-arc}
Let $\CC \subseteq K$ be a field extension. A \textit{$K$-arc} is a morphism of $\CC$-schemes $\Spec K[[t]] \to X$.
\end{definition}

If $\mu:X' \rightarrow X$ is a morphism of schemes, then we have an induced morphism $\mu_{\infty}:X'_{\infty} \rightarrow X_{\infty}$ sending $\gamma$ to $\mu \circ \gamma$.

Let $\gamma: \Spec K[[t]] \to X$ be a $K$-arc on $X$. Let $x = \gamma(o)$. Given an ideal sheaf $\mathfrak{a}$ on X, we define $\ord_{\gamma}(\mathfrak{a})=\displaystyle \min_{f\in \mathfrak{a}_x} \ord_\gamma(f)$. For a nonnegative integer $p$, we define $\Cont^{\geq p}(\mathfrak{a})$, the \textit{contact locus} of $\mathfrak{a}$ of order $p$, to be the closed subscheme of $X_\infty$ whose $K$-valued points (where $\CC \subseteq K$ is an extension of fields) are \begin{equation}\label{defn-of-contact}
\Cont^{\geq p}(\mathfrak{a})(K) = \{\gamma: \Spec K[[t]] \to X \mid \ord_{\gamma}(\mathfrak{a})\geq p \}.
\end{equation}

If $Z$ is a closed subscheme of $X$ defined by the ideal sheaf $\mathcal{I}$, we  write $\Cont^{\geq p}(Z)$ for $\Cont^{\geq p}(\mathcal{I})$. If a closed subscheme structure on a closed subset of $X$ has not been specified, we implicitly give it the reduced subscheme structure. 

For an ideal $\mathfrak{a}$ of $\widehat{\mathcal{O}}_{X, \gamma(o)}$, we define $\ord_\gamma(\mathfrak{a})=\displaystyle \min_{f\in \mathfrak{a}} \ord_\gamma(f)$. For $x \in X$ and an ideal $\mathfrak{a}$ of $\widehat{\mathcal{O}}_{X, x}$ we have a closed subscheme $\Cont^{\geq p}(\mathfrak{a})$ of $X_\infty$ whose $K$-valued points (where $\CC \subseteq K$ is an extension of fields) are  
\begin{equation}\label{defn-of-contact2}
\Cont^{\geq p}(\mathfrak{a})(K) = \{\gamma: \Spec K[[t]] \to X \mid \gamma(o)=x,\ \ord_{\gamma}(\mathfrak{a})\geq p \}.
\end{equation}

%We will often abuse notation and write $\Cont^{\geq p}(\mathfrak{a})$ for $\Cont^{\geq p}(\mathfrak{a})(\CC)$.

\begin{proposition}\label{closed-point}
Let $X$ be a variety over a field $\CC$. Let $\gamma:\Spec \CC[[t]] \to X$ be a $\CC$-arc. Then $\gamma(o) \in X$ is a closed point of $X$ with residue field $\CC$.
\end{proposition}

\begin{proof}
Set $p=\gamma(o)$, and let $\kappa(p)$ denote the residue field of $p \in X$. We have a local $\CC$-algebra homomorphism $\gamma^*:\mathcal{O}_{X,p} \to \CC[[t]]$. Taking the quotient by the maximal ideals, we get a $\CC$-algebra homomorphism $\kappa(p) \hookrightarrow \CC$ that is an isomorphism on $\CC \subseteq \kappa(p).$ Hence $\kappa(p)=\CC$. Since $\trdeg_\CC \kappa(p)=0$, it follows that $p$ is a closed point. 
\end{proof}

\subsection{Points of the arc space}

We next make a couple of remarks about the notion of a \textit{point of the arc space}.

\begin{remark}\label{point}
Let $X$ be a scheme of finite type over a field $\CC$. Let $\alpha \in X_\infty$ be a (not necessarily closed) point of the scheme $X_\infty$. That is, in some open affine patch of $X_\infty$, $\alpha$ corresponds to a prime ideal. Let $\kappa(\alpha)$ denote the residue field at the point $\alpha$ of the scheme $X_\infty$. There is a canonical morphism $\Theta_\alpha:\Spec \kappa(\alpha) \to X_\infty$ induced by the canonical $\CC$-algebra homomorphism $\mathcal{O}_{X_\infty, \alpha} \to \kappa(\alpha)$. By Equation \ref{arc-functor}, the morphism $\Theta_\alpha$ corresponds to a $\kappa(\alpha)$-arc $\theta_\alpha:\Spec \kappa(\alpha)[[t]] \to X$. We will abuse notation and refer to this arc $\theta_\alpha:\Spec \kappa(\alpha)[[t]] \to X$ by $\alpha:\Spec \kappa(\alpha)[[t]] \to X$. That is, given a point $\alpha \in X_\infty$, we have a canonical $\kappa(\alpha)$-arc $\alpha:\Spec \kappa(\alpha)[[t]] \to X$.    
\end{remark}

\begin{remark}
We now examine the reverse of the construction given in Remark \ref{point}. Let $\CC \subseteq K$ be some extension of fields. Given a $K$-arc $\theta:\Spec K[[t]] \to X$, by Equation \ref{arc-functor}, we get a morphism $\Theta:\Spec K \to X_\infty$. The image $\Theta(\mathtt{pt})$ of the only point $\mathtt{pt}$ of $\Spec K$ is a point of $X_\infty$, call it $\alpha$. By Remark \ref{point}, we associate to $\alpha$ a $\kappa(\alpha)$-arc $\Theta_\alpha:\Spec \kappa(\alpha)[[t]] \to X$. Note that $\Theta:\Spec K \to X_\infty$ factors through $\Theta_\alpha:\Spec \kappa(\alpha) \to X_\infty$, since on the level of rings, the $\CC$-algebra map $\Theta^*:\mathcal{O}_{X_\infty, \alpha} \to K$ induces a map $\kappa(\alpha) \to K$. Hence $\theta:\Spec K[[t]] \to X$ factors through $\theta_\alpha:\Spec \kappa(\alpha)[[t]] \to X$. To summarize, $K$-arcs on $X$ correspond to $K$-valued points of $X_\infty$. To each $K$-valued point of $X_\infty$, we can assign a point of $X_\infty$.  If we let $K$ range over all field extensions on $\CC$, this assignment is surjective onto the set of points of $X_\infty$, but it is not injective. To a point $\alpha$ of $X_\infty$, we assign (as described in Remark \ref{point}) a canonical $\kappa(\alpha)$-valued point of $X_\infty$. The point of $X_\infty$ that we assign to this $\kappa(\alpha)$-valued point is $\alpha$.      
\end{remark}

\begin{remark}\label{arcs}
Let $p$ be a closed point of an $n$-dimensional nonsingular variety $X$, and fix generators $x_1, \ldots, x_n$ of the maximal ideal of $\mathcal{O}_{X, p}$. Let $\CC \subseteq K$ be an extension of fields. Giving an arc $\gamma: \Spec K[[t]] \to X$ such that $\gamma \in \Cont^{\geq 1}(p)(K)$ is equivalent to giving a homomorphism of $\CC$-algebras $\widehat{\mathcal{O}}_{X, p}\simeq \CC[[x_1, \ldots, x_n]] \to K[[t]]$ sending each $x_i$ into $(t)K[[t]]$. 
\end{remark}

\begin{definition}\label{zero-arcs}
We say an arc $\gamma: \Spec K[[t]] \to X$ is a \textit{trivial arc} if the maximal ideal of $\widehat{\mathcal{O}}_{X, \gamma(o)}$ equals the kernel of the map $\gamma^*:\widehat{\mathcal{O}}_{X, \gamma(o)} \to K[[t]]$. 
\end{definition}
We have the following observation (whose proof we leave to the reader).

\begin{lemma}\label{lemma2}
Let $X$ be a nonsingular variety. If $\mu: X' \rightarrow X$ is the blowup of a closed point $p \in X$, with exceptional divisor $E$, then:  
\begin{enumerate}
\item \label{lem2-i} Let $\gamma:\Spec K[[t]] \to X$ be an arc such that $\gamma \in \Cont^{\geq 1}(p)$, and suppose $\gamma$ is not the trivial arc. Then there exists a unique arc $\gamma ':\Spec K[[t]] \to X'$ lifting $\gamma$, i.e. $\gamma = \mu \circ \gamma '$. Furthermore, $\gamma '\in \Cont^{\geq 1}(E)$. 

\item \label{lem2-i-2} If $\gamma$ is as in part \ref{lem2-i} and additionally $K=\CC$, then the residue field at $\gamma '(o) \in X'$ equals $\CC$. Furthermore, if $\ord_\gamma(x_1) \leq \ord_\gamma(x_i)$ for all $2 \leq i \leq n$, then there exist $c_i \in \CC$ (for $2 \leq i \leq n$) such that $x_1$ and $\frac{x_i}{x_1}-c_i$ for $2 \leq i\leq n$  are local algebraic coordinates at $\gamma '(o)$.

\item \label{lem2-ii} $\mu_{\infty}(\Cont^{\geq 1}(E))=\Cont^{\geq 1}(p)$.
\end{enumerate}
\end{lemma}

We now describe a geometric construction, called the sequence of centers of a valuation, that is useful in studying valuations, especially those on smooth surfaces. We give the definition only for valuations given by the order of vanishing along an arc $\gamma:\Spec \CC[[t]] \to X$, as this is the case we will be interested in. For a general valuation, the definition is similar \cite[Exer. II.4.12]{hartshorne}.

\begin{definition}[Sequences of centers of an arc valuation]\label{divisor-limit}
Let $X$ be a nonsingular variety over a field $\CC$ and let  $\gamma:\Spec \CC[[t]] \to X$ be an arc that is not a zero arc. The point $p_0:=\gamma(o)$ is called the \textit{center} of $v$ on $X$. We blow-up $p_0$ to get a model $X_1$ with exceptional divisor $E_1$. By Lemma \ref{lemma2}  the arc $\gamma$ has a unique lift to an arc $\gamma_1:\Spec \CC[[t]] \to X_1$. Let $p_1$ be the closed point $\gamma_1(o)$. We define inductively a sequence of closed points $p_i$ and exceptional divisors $E_i$ on models $X_i$, and lifts $\gamma_i:\Spec\CC[[t]] \to X_i$ of $\gamma$ as follows. We blow-up $p_{i-1} \in X_{i-1}$ to get a model $X_i$, and let $E_i$ be the exceptional divisor of this blowup. Let $\gamma_i:\Spec \CC[[t]] \to X_i$ be the lift of $\gamma_{i-1}:\Spec \CC[[t]] \to X_{i-1}$. We denote by $p_i$ the closed point $\gamma_i(o)$, and by $\mu_i:X_i \to X$ the composition of the first $i$ blowups. We call $\{p_i\}_{i \geq 0}$ the \textit{sequence of centers} of $\gamma$. 
\end{definition}

\section{Arc valuations}\label{arcs-valns}
In this section, we begin the study of arc valuations, which are the central object of this paper. We begin by defining arc valuations, normalized arc valuations, and nonsingular arc valuations.

\begin{definition}[Arc valuations]\label{defn-arc-valn}
Let $X$ be a variety over a field $\CC$, and let $p \in X$ be a (not necessarily closed) point. Let $\CC \subseteq K$ be an extension of fields. A  \textit{K-arc valuation} $v$ on $X$ centered at $p$ is a map $v:\mathcal{O}_{X,p} \to \ZZZ$ such that there exists an arc $\gamma:\Spec K[[t]] \to X$  with $\gamma(o)=p$ (where $o$ is the closed point of $\Spec K[[t]]$) and $v(f)=\ord_\gamma(f)$  for $f \in \mathcal{O}_{X,p}$. Since $\ord_\gamma$ extends uniquely to $\widehat{\mathcal{O}}_{X, p}$ (the completion of $\mathcal{O}_{X, p}$ at its maximal ideal), we can extend $v$ to $\widehat{\mathcal{O}}_{X, p}$ as well. This extension does not depend on the choice of arcs $\gamma$ satisfying $v=\ord_\gamma$ on $\mathcal{O}_{X, p}$. Therefore we will also regard arc valuations as maps $v:\widehat{\mathcal{O}}_{X,p} \to \ZZZ$ without additional comment.
\end{definition}

It is shown in \cite[Proposition 2.11]{ishii} that every divisorial valuation is an arc valuation. 

\begin{definition}[Normalized arc valuations]\label{normalized}
We call an arc valuation $v$ centered at a point $p \in X$ \textit{normalized} if the set $\{v(f) \mid f \in \widehat{\mathcal{O}}_{X,p}, \ 0<v(f)<\infty\}$ is non-empty and the greatest common factor of its elements is $1$. Every arc valuation taking some value strictly between $0$ and $\infty$ is a scalar multiple of a normalized valuation. We say an arc $\gamma: \Spec K[[t]] \to X$ is \textit{normalized} if $\ord_\gamma: \widehat{\mathcal{O}}_{X, \gamma(o)} \to \ZZZ$ is a normalized arc valuation. 
\end{definition}

\begin{notation}\label{A-algebra}
Let $X$ be a nonsingular variety over an algebraically closed field $\CC$ of characteristic zero. Let $\gamma:\Spec \CC[[t]] \to X$ be an arc centered at $p \in X$ and let $\gamma^*:\widehat{\mathcal{O}}_{X, p} \to \CC[[t]]$ be the corresponding $\CC$-algebra morphism. Assume $\gamma$ is not a zero arc. Define a $\CC$-algebra $A_\gamma$ by $A_\gamma = \widehat{\mathcal{O}}_{X, p}/\ker(\gamma^*)$. Let $\tilde{A_\gamma}$ be the normalization of $A_\gamma$. Then $\gamma^*$ induces an injective $\CC$-algebra map $\overline{\gamma}^*:A_\gamma \hookrightarrow \CC[[t]]$ which extends to an injective $\CC$-algebra homomorphism $\overline{\gamma}^*:\tilde{A_\gamma} \hookrightarrow \CC[[t]]$. We denote by $\ord_{\overline{\gamma}}$ the composition $\ord_t \circ \overline{\gamma}^*:\tilde{A_\gamma} \to \ZZ_{\geq 0}$. Note that for $f \in \widehat{\mathcal{O}}_{X,p} \setminus \ \ker(\gamma^*)$, we have $\ord_\gamma(f)=\ord_{\overline{\gamma}}(\overline{f})$. We will show in Lemma \ref{hochster-fix2} that there exists $\phi \in \CC[[t]]$ such that the image of $\overline{\gamma}^*:\tilde{A_\gamma} \hookrightarrow \CC[[t]]$ equals $\CC[[\phi]] \subseteq \CC[[t]]$. 

\end{notation}

\begin{lemma}\label{hochster-fix1}
We use the setup described in Notation \ref{A-algebra}. The ring homomorphism $\overline{\gamma}^*:A_\gamma \hookrightarrow \CC[[t]]$ gives $\CC[[t]]$ the structure of a finite $A_\gamma$-module. In particular, $A_\gamma$ has Krull dimension one.
\end{lemma}

\begin{proof}
Choose local coordinates $x_1, \ldots, x_n$ at $p$ such that $\gamma^*(x_1) \neq 0$. We have $\gamma^*(x_1)=t^ru$  for some positive integer $r$ and unit $u \in \CC[[t]]$. Since $\CC$ is algebraically closed and has characteristic zero, there exists a unit $v \in \CC[[t]]$ such that $v^r = u$. Indeed, we may use the binomial series and take $v=u^{1/r}$. To be precise, write $u=u_0(1 + u_1(t))$, with $u_1(t) \in (t)\CC[[t]]$ and $u_0 \neq 0$. Then $u^{1/r}=u_0^{1/r}(1+u_1(t))^{1/r}=u_0^{1/r}(1+ \sum_{i \geq 1} \binom{1/r}{i}u_1^i)$, where $u_0^{1/r}$ denotes any root of $X^r - u_0=0$.

Let $\tau:\CC[[t]] \to \CC[[t]]$ be the $\CC$-algebra automorphism of $\CC[[t]]$  defined by $\tau(t)=tv^{-1}$. Then $\tau(\gamma^*(x_1))=\tau(t^ru)=t^rv^{-r}u=t^r$. Therefore, we may assume without loss of generality that $\gamma^*(x_1)=t^r$. 

I claim $1, t, \ldots, t^{r-1}$ generate $\CC[[t]]$ as a module over $A_\gamma$. Let $f(t)=\sum_{i \geq 0}f_it^i \in \CC[[t]]$ with $f_i \in \CC$ for all $i \geq 0$. For $0\leq j \leq r$, define a power series $p_j(X) \in \CC[[X]]$ by  $p_j(X)=\sum_{i\geq 0} f_{j+ ir}X^i$. Then 
\begin{align*}
\sum_{j=0}^{j=r-1} \gamma^*(p_j(x_1))t^j&=\sum_{j=0}^{j=r-1} p_j(\gamma^*(x_1))t^j \\
&=\sum_{j=0}^{j=r-1} p_j(t^r)t^j \\
&=\sum_{j=0}^{j=r-1}\sum_{i\geq 0}f_{j+ir}t^{j+ir} = \sum_{i \geq 0} f_it^i = f(t). 
\end{align*}

Hence $1, t, \ldots, t^{r-1}$ generate $\CC[[t]]$ considered as a module over $A_\gamma$ via the ring homomorphism $\overline{\gamma}^*:A_\gamma \hookrightarrow \CC[[t]]$. Since $\CC[[t]]$ has dimension one and module finite ring extensions preserve dimension (\cite[Proposition 9.2]{E}), we conclude $A_\gamma$ has dimension one.
\end{proof}

%\begin{lemma}
%Let $u(t), f(t) \in \CC[[t]]$ with $u(t)$ a unit in $\CC[[t]]$, and let $r$ be a positive integer. There exists power series $p_0(X), \ldots, p_{r-1}(X) \in \CC[[X]]$ such that $f(t) = \sum_{i=0}^{i=r-1} p_i(t^ru)t^i$.
%\end{lemma}

%\begin{proof}
% and . For $0 \leq i \leq r-1$, set $p_i(X)=f_i$. 
%We need $\sum_{i \geq r}f_it^i=p_r(t^ru)t^r$.
%\end{proof}

\begin{lemma}\label{hochster-fix2}
We continue using the setup and hypotheses of Lemma \ref{hochster-fix1}. There exists $\phi \in \CC[[t]]$ such that the image of $\overline{\gamma}^*:\tilde{A_\gamma} \hookrightarrow \CC[[t]]$ equals $\CC[[\phi]] \subseteq \CC[[t]]$.
\end{lemma}

\begin{proof}
Since an integral extension of rings preserves dimension (\cite[Proposition 9.2]{E}), we have that $\tilde{A}_\gamma$ has dimension one.
Since $\CC[[t]]$ is normal (in fact it is a DVR), the local $\CC$-algebra map $\overline{\gamma}^*:A_\gamma \hookrightarrow \CC[[t]]$ extends to a $\CC$-algebra map $\overline{\gamma}^*:\tilde{A_\gamma} \hookrightarrow \CC[[t]]$.  

I claim the ring $\tilde{A_\gamma}$ is a complete local domain. The local ring $A_\gamma$ is complete since it is the image of a complete local ring. The normalization of an excellent ring $A$ (in our case, the complete local domain $A_\gamma$) is module finite over $A$ \cite[p.259]{M}. A module finite domain over a complete local domain is local and complete (apply \cite[Corollary 7.6]{E} and use the domain hypothesis to conclude there is only one maximal ideal). Hence $\tilde{A_\gamma}$ is a complete local domain.
  
Since $\tilde{A_\gamma}$ is a complete normal 1-dimensional local domain containing the field $\CC$, it is isomorphic to a power series over $\CC$ in one variable \cite[Cor. 2, p.206]{M}. That is, there exists $\phi \in \CC[[t]]$ such that the image of $\overline{\gamma}^*:\tilde{A_\gamma} \hookrightarrow \CC[[t]]$ equals $\CC[[\phi]] \subseteq \CC[[t]]$. 
\end{proof}

The following result was pointed out to me by Mel Hochster.

\begin{proposition}\label{hochster-genius}
Assume the setup of  Notation \ref{A-algebra} and let $\phi$ be as in Lemma \ref{hochster-fix2}. Let $d$ be the greatest common divisor of the elements of the non-empty set $\{\ord_\gamma(f) \mid f \in \widehat{\mathcal{O}}_{X,p}, \ 0<\ord_\gamma(f)<\infty\}$. Then  $d=\ord_t(\phi)$. In particular, $\ord_\gamma$ is a normalized arc valuation if and only if $\ord_t(\phi)=1$.
\end{proposition}

\begin{proof}
For $f, g \in A_\gamma$ such that $\frac{f}{g} \in \tilde{A_\gamma} \subseteq \Frac(A_\gamma)$, we have $\ord_{\overline{\gamma}}(\frac{f}{g})=\ord_{\overline{\gamma}}(f)-\ord_{\overline{\gamma}}(g)$, and hence $d$ divides $\ord_{\overline{\gamma}}(\frac{f}{g})$. In particular $d$ divides $\ord_t(\phi)$. We have $\overline{\gamma}^*(A_\gamma) \subseteq \overline{\gamma}^*(\tilde{A_\gamma}) = \CC[[\phi]] \subseteq \CC[[t]]$ and hence $\ord_t(\phi)$ divides $\ord_\gamma(f)$ for all $f$ $\in$ $A_\gamma$. So $\ord_t(\phi)$ divides $d$. Hence $d=\ord_t(\phi)$.
\end{proof}

\begin{definition}[Nonsingular arc valuations]\label{nonsingular}
Let $v$ be an arc valuation centered at $p$, and let $\mathfrak{m}_p$ denote the maximal ideal of $\widehat{\mathcal{O}}_{X, p}$. We call $v$ \textit{nonsingular} if 
\begin{equation}
\displaystyle \min_{f \in \mathfrak{m}_p} v(f) \ = 1.
\end{equation}
If $\gamma \in X_\infty$, then we say $\gamma$ is nonsingular if $\ord_\gamma$ is a nonsingular valuation.
\end{definition}

Let $C$ be an irreducible subset of $X_\infty$, and let $\alpha$ be the generic point of $C$. Following Ein, Lazarsfeld, and Musta\c{t}\v{a} \cite[p.3]{mustata}, we define a map $\val_C:\mathcal{O}_{X, \alpha(o)} \to \ZZZ$ by setting for $f \in \mathcal{O}_{X,\alpha(o)}$
\begin{equation}\label{valc} 
\val_C(f)=\displaystyle \min \{ \ord_\gamma(f)  \mid \gamma\in C \   \textrm{such that} \  f \in \mathcal{O}_{X, \gamma(o)}\}
\end{equation}

\begin{proposition}\label{gen-pt-cyl}
If $C \subseteq X_\infty$ is an irreducible subset with generic point $\alpha:\Spec K[[t]] \to X$, then $\val_C = \ord_\alpha$ on $\mathcal{O}_{X, \alpha(o)}$. In particular, $\val_C$ is an arc valuation.
\end{proposition}

\begin{proof}
Fix $f \in \mathcal{O}_{X, \alpha(o)}$, and let $U \subseteq X$ be the maximal open set on which $f$ is regular. We have $\ord_\alpha(f) \geq \val_{C}(f)$ by Equation (\ref{valc}). Let $\alpha ' \in C$ be such that $\val_{C}(f)=\ord_{\alpha '}(f)$. Let $\pi:X_\infty \to X$ be the canonical morphism sending $\gamma \to \gamma(o)$. If $\ord_\alpha(f)>\val_C(f)$, then $C \cap \Cont^{\geq \ord_\alpha(f)}(f)$ is a closed subset of the irreducible set $C \cap \pi^{-1}(U)$, containing $\alpha$ but not $\alpha ' \in C$, contradicting $\overline{\{\alpha\}}=C$. Hence $\ord_\alpha(f)=\val_C(f)$ for all $f \in \mathcal{O}_{X,\alpha(o)}$.

\end{proof}

Next, we show arc valuations are the same as $\ZZZ$-valued valuations, which are defined as follows:

\begin{definition}\label{defn-of-valn}
Let $R$ be a $\CC$-algebra. A $\ZZZ$\textit{-valued valuation} on $R$ is a map $v:R \to \ZZZ$ such that
\begin{enumerate}
\item $v(c)=0$ for $c \in \CC^*$
\item $v(0)=\infty$
\item $v(xy)=v(x)+v(y)$ for $x, y \in R$
\item $v(x+y) \geq \min \{v(x), v(y)\}$ for $x, y \in R$
\item $v$ is not identically $0$ on $R^*$.
\end{enumerate}  
\end{definition} 

Let $p \in X$ be a (not necessarily closed) point of $X$, and let $v:\mathcal{O}_{X,p} \to \ZZZ$ be a $\ZZZ$-valued valuation. Set $\mathfrak{p}=\{f \in \mathcal{O}_{X,p} \mid v(f)=\infty\}$. We have an induced valuation $\tilde{v}:\mathcal{O}_{X,p}/\mathfrak{p} \to \ZZ$ that extends as usual to a valuation $\tilde{v}:\Frac(\mathcal{O}_{X,p}/\mathfrak{p}) \to \mathbb{Z}$. Let $R_{\tilde{v}} = \{f \in \Frac(\mathcal{O}_{X,p}/\mathfrak{p}) \mid \tilde{v}(f) \geq 0\}$ be the valuation ring of $\tilde{v}$. $R_{\tilde{v}}$ is a discrete valuation ring. Let $\mathfrak{m}_{\tilde{v}}$ be the maximal ideal of $R_{\tilde{v}}$, and let $\kappa(v) = R_{\tilde{v}}/\mathfrak{m}_{\tilde{v}}$.

\begin{proposition}
Let $p \in X$ be a (not necessarily closed) point of $X$. If $v:\mathcal{O}_{X,p} \to \ZZ_{\geq 0} \cup \{\infty\}$ is a valuation as in Definition \ref{defn-of-valn}, then $v$ is an arc valuation on $X$.
\end{proposition}

\begin{proof}
 The completion $\widehat{R}_{\tilde{v}}$ of $R_{\tilde{v}}$ with respect $\mathfrak{m}_{\tilde{v}}$ is again a discrete valuation ring. Hence $\widehat{R}_{\tilde{v}}$ is isomorphic to the power series ring $\kappa(v)[[t]]$. The composition of the canonical homomorphisms $\mathcal{O}_{X,p} \to \mathcal{O}_{X,p}/\mathfrak{p} \to R_{\tilde{v}} \to \widehat{R}_{\tilde{v}} = \kappa(v)[[t]]$ gives an arc $\gamma: \Spec \kappa(v)[[t]] \to X$. Tracing through the constructions, we see that $\ord_\gamma = v$ on $\mathcal{O}_{X,p}$.   
\end{proof}

\begin{remark}\label{krull} 
Following \cite{mattias}, a \textit{Krull valuation} $V$ is a map $V:\CC(X)^* \to \Gamma$, where $\CC(X)$ is the function field of $X$ and $\Gamma$ is a totally ordered abelian group, satisfying equations (1), (3), (4), (5) of Definition \ref{defn-of-valn}. For a discussion of the differences between Krull valuations and valuations (as defined in Definition \ref{defn-of-valn}) in the case of surfaces, see \cite[Section 1.6]{mattias}. For example, Favre and Jonsson associate to any Krull valuation $V:\mathbb{C}[[x, y]] \to \Gamma$ other than an exceptional curve valuation, a unique (up to scalar multiple) valuation $v:\mathbb{C}[[x, y]] \to \mathbb{R}\cup \{\infty\}$ \cite[Prop. 1.6]{mattias}. 
\end{remark}

To any Krull valuation $V:\CC(X)^* \to \ZZ^r$ (where $\ZZ^r$ is lexicographically ordered with $(0, \ldots, 0, 1)$ as the smallest positive element) with center $p$ (that is, the valuation ring $R_V:=\{f \in \CC(X)^* \mid V(f) \geq 0\} \cup \{0\}$ dominates $\mathcal{O}_{X,p}$), we associate an arc valuation $v:\mathcal{O}_{X,p} \to \ZZ_{\geq 0} \cup \{\infty\}$ as follows. Set $v(0)=\infty$. For $f \in \mathcal{O}_{X,p}$, suppose $V(f)=(a_1, \ldots, a_r)$. If $a_1=a_2=\ldots=a_{r-1}=0$, set $v(f)=a_r$. Otherwise, set $v(f)=\infty$. 

When $\dim X=2$, the above association $V \to v$ gives a bijection between Krull valuations $V:\CC(X)^* \to \ZZ^2$ centered at $p$ and arc valuations centered at $p$ \cite[Prop. 1.6]{mattias}.

The following example shows this association $V \to v$ is not injective in general.

\begin{example}\label{bad-example}
Let $X=\Spec \CC[x, y, z]$ and let $V_1:\CC(X)^* \to \ZZ^2$ and $V_2:\CC(X)^* \to \ZZ^3$ be Krull valuations defined by 
$V_1(\sum c_{ijk}x^iy^jz^k = \min\{(j+2k,i) \mid c_{ijk} \neq 0\}$ and
$V_2(\sum c_{ijk}x^iy^jz^k = \min\{(0,j+k,i) \mid c_{ijk} \neq 0\}$. Then $V_1, V_2$ both have transcendence degree $0$ over $\CC$, and have the same sequence of centers. The arc valuations associated (in the manner described above) to $V_1$ and $V_2$  both equal the arc valuation $\ord_\gamma$ where $\gamma:\Spec \CC[[t]] \to X$ is the arc given by $x \to t$, $y \to 0$, and $z \to 0$.
\end{example}

\section{Desingularization of normalized $\CC$-arc valuations}\label{zero-valns}

In this section, we prove that a normalized $\CC$-arc valuation on a nonsingular variety $X$ over a field $\CC$ can be desingularized. Specifically, the goal of this section is to prove Proposition \ref{val-type}, which says that a normalized $\CC$-arc can be lifted after finitely many blowups to a $\CC$-arc that is nonsingular. Our proof is based on Hamburger-Noether expansions.
 
Let $X$ be a nonsingular variety of dimension $n$ ($n \geq 2$) over a field $\CC$ and let $p_0 \in X$ be a closed point. Let $\gamma:\Spec \CC[[t]] \to X$ be an arc such that $\gamma(o)=p_0$ and $v:=\ord_\gamma$ is a normalized arc valuation (Definition \ref{normalized}). Let $p_i \in X_i$ ($i\geq 0$) be the sequence of centers of $v$, as described in Definition \ref{divisor-limit}.  If $\gamma_r$ denotes the unique lift of $\gamma$ to $X_r$ (by Lemma \ref{lemma2}), then note that $v$ extends to the valuation $\widehat{\mathcal{O}}_{X_r,p_r} \to \ZZZ$ associated to $\gamma_r$. Hence for $f \in \widehat{\mathcal{O}}_{X_r, p_r}$, we will write $v(f)$ to mean $\ord_{\gamma_r}(f)$.

\subsection{Hamburger-Noether expansions}\label{hne-section}
We will use a list of equations known as Hamburger-Noether expansions (HNEs) to keep track of local coordinates of the sequences of centers of $v$. We explain HNEs in this section. Our source for this material is \cite[Section 1]{delgado}, where the presentation is given for arbitrary valuations on a nonsingular surface. 

HNEs are constructed by repeated application of Lemma \ref{lemma2} part \ref{lem2-i-2}, which we recall:

\begin{lemma}\label{basis-of-hne}
Let $X$ be a nonsingular variety of dimension $n$ ($n \geq 2$) over a field $\CC$ and let $p_0 \in X$ be a closed point. Let $\gamma:\Spec \CC[[t]] \to X$ be an arc such that $\gamma(o)=p_0$ and $v:=\ord_\gamma$ is a normalized arc valuation (Definition \ref{normalized}).  Let $x_1, x_2, \ldots , x_n$ be local algebraic coordinates at $p_0$ such that $1\leq v(x_1) \leq v(x_i)$ for $2 \leq i \leq n$. Then for $2 \leq i \leq n$, there exists $a_{i,1} \in \CC$ such that if we let $y_i=\frac{x_i}{x_1} - a_{i,1} \in \CC(X)$, then $x_1, y_2, \ldots, y_n$   generate the maximal ideal of $\mathcal{O}_{X_1, p_1} \subseteq \CC(X)=\CC(X_1)$.
\end{lemma}

We now describe how to write down the HNEs, following \cite[Section 1]{delgado}. Let $x_i, a_{i, 1}, y_i$ be as in Lemma \ref{basis-of-hne}. We have $x_i=a_{i,1}x_1 + x_1y_i$. If $v(x_1) \leq v(y_i)$ for every $2 \leq i \leq n$, then with the local algebraic coordinates $x_1, y_2, \ldots, y_n$ at $p_1$ we are in a similar situation as before, and we repeat the process of applying Lemma \ref{basis-of-hne} to get local algebraic coordinates at $p_2$. Suppose that after $h$ steps we have local algebraic coordinates $x_1, y_2', \ldots y_n'$ at $p_h$ such that $v(x_1) > v(y_j')$ for some $2 \leq j \leq n$. We may choose $j$ such that $v(y_j') \leq v(y_i')$ for $2 \leq i \leq n$. There are $a_{i,k} \in \CC$ such that 
\begin{equation}\label{hne-eqn}
x_i=a_{i,1}x_1 + a_{i, 2}x^2_1 + \ldots + a_{i,h}x^h_1 + x^h_1y_i'
\end{equation}
 for $2 \leq i \leq n$, $1 \leq k\leq h$. The assumption that $p_h$ is a closed point implies $v(y_i')>0$ for $2 \leq i \leq n$. Let $z_1=y_j'$, and we repeat the procedure of applying Lemma \ref{basis-of-hne} with the local coordinates $z_1, x_1, y_2', \ldots, y_{j-1}',y_{j+1}', \ldots, y_n'$ (note that we brought $z_1$ to the front of the list because it is the coordinate with smallest value). We will refer to such a change in the first coordinate (in this case, from $x_1$ to $z_1$) of our list as an iteration. 

If we do not arrive at a situation where $v(x_1)>v(y_j')$ for some $2 \leq j \leq n$, then there exist $a_{i,k} \in \CC$ (for $2\leq i\leq n$, and all $k \geq 1$) such that $$v\left(\frac{x_i-\sum_{k=1}^Na_{i,k}x^k_1}{x^N_1}\right) \geq v(x_1),$$ and hence (since $v(x_1) \geq 1$)

\begin{equation}\label{hne-eqn2}
v\left( x_i-\sum_{k=1}^{N} a_{i, k}x^k_1\right)>N
\end{equation}
for all $N>0.$

Let $z_0=x_1$, and for $l>0$ let $z_l$ be the first listed local coordinate at the $l$-th iteration. We have $v(z_l)<v(z_{l-1})$ since an iteration occurs when the smallest value of the local coordinates at the center decreases in value after a blowup. So $\{v(z_l)\}_{l \geq 0}$ is a strictly decreasing sequence of positive integers, and hence must be finite, say $v(z_0), v(z_1), \ldots, v(z_L)$. %The assumption that the valuation $v$ has been normalized implies $v(z_L)=1$. 

For notational convenience, redefine $x_1, \ldots, x_n$ to be the local algebraic coordinates after the final iteration, with $x_1=z_L$. So $x_1, \ldots, x_n$ are local algebraic coordinates centered at $p_r$ on $X_r$ for some $r$, and Equation \ref{hne-eqn2} becomes 
\begin{equation}\label{hne-eqn3}
v(x_i-\sum_{k=1}^{N} c_{i, k}x^k_1)>N
\end{equation}
for $2 \leq i \leq n$, $c_{i, k} \in \CC$, and all $N>0$. 

\begin{definition}\label{defn-of-P}
Let $P_1(t)=t$, and for $2 \leq i \leq n$ define $P_i(t) \in \CC[[t]]$ by $P_i(t)=\sum_{k=1}^{\infty} c_{i,k}t^k$.
\end{definition}

\begin{remark}\label{eqn4}
Equation \ref{hne-eqn3} implies  $v(x_i-P_i(x_1))=\infty$ for $2\leq i\leq n$. 
\end{remark}

%\begin{lemma}
%For $2\leq i\leq n$, we have $v(x_i-P_i(x_1))=\infty$ (where $P_i$ is as in Definition \ref{defn-of-P}). Here we regard $x_i-P_i(x_1)\in \widehat{\mathcal{O}}_{X_r, p_r} \simeq \CC[[x_1, \ldots, x_n]]$ and $v:\widehat{\mathcal{O}}_{X_r, p_r} \to \ZZZ$ (see Definition \ref{defn-arc-valn}).
%\end{lemma}

%\begin{proof}
%Since $x_1, \ldots ,x_n$ are the variables in the last iteration of the HNE, we have $$v(\frac{x_i-\sum_{j=1}^kc_{i,j}x^j_1}{x^k_1}) \geq v(x_1),$$ and hence $v(x_i-\displaystyle \sum_{j=1}^kc_{i,j}x^j_1) \geq (k+1)v(x_1)$. For $k>0$, let $$r_k = \displaystyle\frac{\sum_{j=1}^{k}c_{i,j}x_1^k - P_i(x_1)}{x_1^{k+1}}.$$ We have $r_k \in \CC[[x_1, \ldots, x_n]]$ and 
%\begin{align}
%v(x_i-P_i(x_1))&=v(x_i-\displaystyle\sum_{j=1}^kc_{i,j}x^j_1 + x^{k+1}_1r_k) \\
%&\geq \min \{v(x_i-\displaystyle\sum_{j=1}^kc_{i,j}x^j_1), v(x^{k+1}_1r_k)\} \\
%&\geq  (k+1)v(x_1) = k+1
%\end{align}
%Letting $k \to \infty$ we conclude $v(x_i-P_i(x_1)) = \infty$. 
%\end{proof}

\begin{lemma}\label{formula}
For every $\psi=\psi(x_1, \ldots, x_n) \in \widehat{\mathcal{O}}_{X_r, p_r} \simeq \CC[[x_1,\ldots, x_n]]$, we have $v(\psi) = \ord_t\psi(t, P_2(t), \ldots, P_n(t)).$   
\end{lemma}
\begin{proof}
Since $\CC[[x_1, \ldots, x_n]]/(x_2-P_2(x_1), \ldots,x_n-P_n(x_1)) \simeq \CC[[x_1]]$, we may write $\psi(x_1, \ldots, x_n)=q(x_1)+ \sum_{i=2}^{n}(x_i-P_i(x_1))h_i$ for $h_i\ \in \CC[[x_1, \ldots, x_n]]$ and $q(x_1) \in \CC[[x_1]]$. Note that $q(x_1)=\psi(x_1, P_2(x_1), \ldots, P_n(x_1))$. We have $v(\psi) \geq \min\{v(q), v((x_2-P_2(x_1))h_2), \ldots, v((x_n-P_n(x_1))h_n)\}$. Since $v((x_i-P_i(x_1))h_i)=\infty$, we have $v(\psi)=v(q)$, since in general, if $v(a) \neq v(b)$, then $v(a+b)=\min\{v(a), v(b)\}$. 

Let $n=\ord_{x_1} q(x_1)$. We claim $v(q)=nv(x_1)$. If $n=\infty$, then $q=0$ and both sides of $v(q)=nv(x_1)$ are $\infty$. If $n < \infty$, then $q=x^n_1u$ for a unit $u$ in $\CC[[x_1]]$. We have $v(u)=0$, since $0=v(1)=v(uu^{-1})=v(u) + v(u^{-1})$ and $v(u), v(u^{-1}) \geq 0$. Hence $v(q)=nv(x_1)$.

So we have $v(\psi)=v(q)=(\ord_{x_1}q(x_1))v(x_1)=\ord_{x_1}\psi(x_1, P_2(x_1) \ldots, P_n(x_1))\cdot v(x_1)$. Since $\psi$ was arbitrary, we have that the image of $v:\CC[[x_1, \ldots, x_n]] \to \ZZZ$ equals $\ZZ_{\geq 0}\cdot v(x_1) \cup \{\infty\}$. Since $v$ was normalized so that the image of $v$ had $1$ as the greatest common factor of its elements, we have $v(x_1)=1$ and $v(\psi) = \ord_t\psi(t,P_2(t), \ldots, P_n(t)).$
\end{proof}

Summarizing the discussion so far, we have:

\begin{proposition}\label{val-type}

Let $v$ be a normalized $\CC$-arc valuation on a nonsingular variety $X$ over a field $\CC$. Then there exists a nonnegative integer $r$ and local algebraic coordinates $x_1, \ldots, x_n$ at the center $p_r$ of $v$ on $X_r$ and $$P_i(t) \in (t)\CC[[t]]$$ for $2 \leq i \leq n$ such that for every $\psi=\psi(x_1, \ldots, x_n) \in \widehat{\mathcal{O}}_{X_r, p_r} \simeq \CC[[x_1,\ldots, x_n]]$, we have $$v(\psi) = \ord_t\psi(t,P_2(t), \ldots, P_n(t)).$$

\end{proposition}

Roughly speaking, this result says that a normalized $\CC$-arc valuation can be desingularized. More precisely, a normalized $\CC$-valued arc $\gamma$ can be lifted after finitely many blowups (of its centers) to an arc $\gamma_r$ that is nonsingular (see Definition \ref{nonsingular} for the definition of nonsingular arc). Using the notation of Proposition \ref{val-type}, the arc $\gamma_r:\Spec \CC[[t]] \to X_r$ is given by the $\CC$-algebra map $\widehat{\mathcal{O}}_{X_r, p_r} \to \CC[[t]]$ with $\ord_{\gamma_r}(x_1)=1$ and $x_i \to P_i(\gamma^*_r(x_1))$ for $2 \leq i \leq n$. Since $\ord_{\gamma_r}(x_1)=1$, we have $\gamma_r$ is a nonsingular arc. 

If the arc $\gamma$ is nonsingular, we can take $r=0$ in Proposition \ref{val-type}, and we have the following result.

\begin{proposition}\label{nonsingular-arc-thm}

Let $\gamma:\Spec \CC[[t]] \to X$ be a nonsingular $\CC$-arc on a nonsingular variety $X$ over a field $\CC$. Let  $x_1, \ldots, x_n$ be local algebraic coordinates at $p=\gamma(o)$ on $X$ with $\ord_\gamma(x_1)=1$ (Definition \ref{nonsingular}). Then there exists $$P_i(t) \in (t)\CC[[t]]$$ for $2 \leq i \leq n$ such that $\gamma^*(x_i)=P_i(\gamma^*(x_1))$ for $2 \leq i \leq n$. Furthermore, for every $\psi=\psi(x_1, \ldots, x_n) \in \widehat{\mathcal{O}}_{X, p} \simeq \CC[[x_1,\ldots, x_n]]$, we have $$\ord_\gamma(\psi) = \ord_t\psi(t,P_2(t), \ldots, P_n(t)).$$
\end{proposition}

\begin{proof}
Since $\ord_\gamma(x_1)=1$, there can be no iterations in the Hamburger-Noether algorithm for $v=\ord_\gamma$. Hence Equation \ref{hne-eqn3} holds, and in particular, Remark \ref{eqn4} applies. That is, if the $P_i(t)$ for $2 \leq i \leq n$ are as in Definition \ref{defn-of-P}, we have $\ord_\gamma(x_i-P_i(x_1))=\infty$ for $2 \leq i \leq n$. So $\gamma^*(x_i-P_i(x_1))=0$, and therefore $\gamma^*(x_i)=\gamma^*(P_i(x_1))=P_i(\gamma^*(x_1))$ for $2 \leq i \leq n$. According to Lemma \ref{formula}, for every $\psi=\psi(x_1, \ldots, x_n) \in \widehat{\mathcal{O}}_{X, p} \simeq \CC[[x_1,\ldots, x_n]]$, we have $$\ord_\gamma(\psi) = \ord_t\psi(t,P_2(t), \ldots, P_n(t)).$$
\end{proof}

We will see in the next section that for a nonsingular $\CC$-valued arc $\gamma$, one can explicitly compute the ideals of $\bigcap_{q \geq 1}\overline{\mu_{q \infty}(\Cont^{\geq 1}(E_q))}$ and $\bigcap_{q \geq 1} \Cont^{\geq q}(\mathfrak{a}_q)$, where $\mathfrak{a}_q = \{ f \in \widehat{\mathcal{O}}_{X, \gamma(o)} \mid \ord_\gamma(f) \geq q\}$. We will see that these ideals are the same, and thus these two sets are equal.

\section{Main results:  $\CC$-arc valuations}\label{main}
\subsection{Introduction}
In this section, we present the main results of the paper. Let $X$ be a nonsingular variety of dimension $n$ ($n \geq 2$) over a field $\CC$. Let $v$ be a normalized $\CC$-arc valuation on $X$. We associate to $v$ several different subsets of the arc space $X_\infty$. In notation we will explain later, these subsets are $C(v)$, $\bigcap_{q \geq 1}\mu_{q \infty}(\Cont^{\geq 1}(E_q))$, $ \bigcap_{q \geq 1} \Cont^{\geq q}(\mathfrak{a}_q)$ and $\pi^{-1}(\alpha(o)) \cap Y_\infty$, and $R=\{a \circ h \in X_\infty \mid h:\Spec \CC[[t]] \to \Spec \CC[[t]]\}$. Our main result is that these five subsets are all equal. We first analyze the case when $v$ is a nonsingular arc valuation (Definition \ref{nonsingular}). We then consider the general case where we drop the hypothesis of nonsingularity. 

\subsection{Setup}
For the remainder of this paper, we fix the following notation. Let $X$ be a nonsingular variety of dimension $n$ ($n \geq 2$) over a field $\CC$. Let $\alpha:\Spec \CC[[t]] \to X$ be a normalized arc valuation on $X$ (see Definition \ref{normalized}). Set $v = \ord_\alpha$. 

In Definition \ref{divisor-limit}, we defined the sequence of centers of a $\CC$-arc valuation. To set notation for the rest of this section, we recall this definition.

\begin{definition}[Sequences of centers of a $\CC$-arc valuation]\label{divisor-limit2}
Let $X$ be a nonsingular variety over a field $\CC$.  Let $\alpha:\Spec \CC[[t]] \to X$ be an arc on $X$. Assume $\alpha$ is not the trivial arc (Definition \ref{zero-arcs}). Set $p_0=\alpha(o)$ (where $o$ is the closed point of $\Spec \CC[[t]]$) and $v=\ord_\alpha$. By Proposition \ref{closed-point}, the point $p_0$ is a closed point (with residue field $\CC$) of $X$. The point $p_0$ is called the \textit{center} of $v$ on $X_0:=X$. Blowup $p_0$ to get a model $X_1$ with exceptional divisor $E_1$. By Lemma \ref{lemma2} the arc $\alpha$ has a unique lift to an arc $\alpha_1:\Spec \CC[[t]] \to X_1$. Let $p_1$ be the closed point $\alpha_1(o)$. Inductively define a sequence of closed points $p_i$ and exceptional divisors $E_i$ on models $X_i$ and lifts $\alpha_i:\Spec\CC[[t]] \to X_i$ of $\alpha$ as follows. Blowup $p_{i-1} \in X_{i-1}$, to get a model $X_i$. Let $E_i$ be the exceptional divisor of this blowup. Let $\alpha_i:\Spec \CC[[t]] \to X_i$ be the lift of $\alpha_{i-1}:\Spec \CC[[t]] \to X_{i-1}$. Let $p_i$ be the closed point $\alpha_i(o)$. Let $\mu_i:X_i \to X$ be the composition of the first $i$ blowups. We call $\{p_i\}_{i \geq 0}$ the \textit{sequence of centers} of $v$. 
\end{definition}

%Let $p_i \in X_i$ be the sequence of centers of $v$, and $E_i \subseteq X_i$ the exceptional divisor of the $i$-th blowup, as in Section \ref{divisor-limit2}. Let $\mu_i:X_i \to X$ be the composition of the first $i$ blowups of centers of $v$.

\subsection{Simplified situation}\label{simplified-section}

We first consider the special case when the arc $\alpha:\Spec \CC[[t]] \to X$  is nonsingular (Definition \ref{nonsingular}). 

\begin{proposition}\label{base-case-cyl2}
Let $X$ be a nonsingular variety of dimension $n$ ($n \geq 2$) over a field $\CC$. Let $\alpha:\Spec \CC[[t]] \to X$ a nonsingular arc (Definition \ref{nonsingular}). 
Set $v=\ord_\alpha$ and $p_0=\alpha(o)$. 
Let $C=\bigcap_{q \geq 1} \mu_{q  \infty}(\Cont^{\geq 1}(E_q))$. Then 

\begin{enumerate}
\item $C$ is an irreducible subset of $X_\infty.$
\item Let $\mathfrak{a}_q = \{f \in \widehat{\mathcal{O}}_{X, p_0} \mid v(f) \geq q\}$. Then $C=\bigcap_{q \geq 1} \Cont^{\geq q}(\mathfrak{a}_q).$
\item $\val_C=v$ on $\widehat{\mathcal{O}}_{X, p_0}$.
\end{enumerate}
\end{proposition}

\begin{notation}\label{mess}
Let $\mathfrak{m}$ be the maximal ideal of $\mathcal{O}_{X, p_0}$. Since $\alpha$ is nonsingular, there exists $x_1 \in \mathfrak{m}$ such that $\ord_\alpha(x_1)=1$. Since $\ord_\alpha(x_1)=1$, we have $x_1 \in\mathfrak{m}\setminus\mathfrak{m}^2$. Choose $x_2, \ldots, x_n$ in $\mathfrak{m}$ so that $x_1, \ldots, x_n$ are local algebraic coordinates at $p_0$ (i.e. generators of $\mathfrak{m}$). Apply Proposition \ref{nonsingular-arc-thm} to the nonsingular arc $\alpha$ to get  $P_i(t) \in (t)\CC[[t]]$, for $2 \leq i \leq n$, such that $\alpha^*(x_i)=P_i(\alpha^*(x_1))$ and for every $\psi(x_1, \ldots, x_n) \in \widehat{\mathcal{O}}_{X, p_0}\simeq\CC[[x_1, \ldots, x_n]]$, we have 
\begin{equation}\label{calc-valn}
v(\psi) =\ord_t\psi(t,P_2(t), \ldots, P_n(t)).
\end{equation}

Write $P_i(t) =\displaystyle \sum_{j \geq 1}c_{i,j}t^j \in (t)\CC[[t]]$ for $2 \leq i\leq n$ and $c_{i,j} \in \CC$.  For $2 \leq i \leq n$, we have 
\begin{align}%\label{grp}
\alpha^*(x_i)&=P_i(\alpha^*(x_1)) \notag \\
&=\displaystyle \sum_{j \geq 1}c_{i,j}(\alpha^*(x_1))^j \label{positive-value}
\end{align}
\end{notation}

We break up the proof of Proposition \ref{base-case-cyl2} into several steps. For the remainder of this section,  $v$, $x_1, \ldots, x_n$, $P_2(t), \ldots, P_n(t)$ and $c_{i,j}$ are as in Proposition \ref{base-case-cyl2} and Notation \ref{mess}.

\begin{lemma}\label{local-coords}
With the notation in Definition \ref{divisor-limit2}, Proposition \ref{base-case-cyl2}, and Notation \ref{mess}, the functions $x_1$ and $\frac{x_i-c_{i,1}x_1-c_{i,2}x^2_1  \cdots - c_{i,q-1}x^{q-1}_1}{x^{q-1}_1} \in \CC(X)$ for $2 \leq i\leq n$ form local algebraic coordinates on $X_{q-1}$ centered at $p_{q-1}$.
\end{lemma}

\begin{proof}
These $n$ functions are elements of positive value under $\ord_{\alpha_q}$ (by Equation \ref{positive-value}), and hence lie in the maximal ideal of the $n$-dimensional regular local ring $\mathcal{O}_{X_{q-1}, p_{q-1}}$. The ideal $\mathfrak{n} \subseteq \mathcal{O}_{X_{q-1},p_{q-1}}$ they generate satisfies $\mathcal{O}_{X_{q-1}, p_{q-1}} /\mathfrak{n} \simeq \CC$, and hence $\mathfrak{n}$ is a maximal ideal.
\end{proof}

\subsection{Reduction to $X=\mathbb{A}^n$}

We denote the affine line $\mathbb{A}^1_\CC=\Spec \CC[T]$ simply by $\mathbb{A}^1$. We show that we may reduce many computations about the arc space of the nonsingular $n$-dimensional variety $X$ to the case $X=\mathbb{A}^n$.  

\begin{proposition}\label{reduction}
Let $X$ be a nonsingular variety and $p \in X$. Let $\pi: X_\infty \to X$ be the canonical morphism sending an arc $\gamma$ to its center $\gamma(o)$. Then $\pi^{-1}(p) \simeq (\mathbb{A}_{\kappa(p)}^n)_\infty$, where $\kappa(p)$ is the residue field at $p \in X$. In particular, if $\kappa(p)=\CC$ then $\pi^{-1}(p) \simeq (\mathbb{A}^n)_\infty.$
\end{proposition}

\begin{proof}
Since  $X$ is nonsingular, there exists an open affine neighborhood $U$ of $p$ and an \'etale morphism $\phi:U \to \Spec \CC[X_1, \ldots, X_n] = \mathbb{A}^n$  (\cite[Prop. 3.24b]{milne}). We will use the following fact (\cite[p.7]{jet-bible}): if $f:X \to Y$ is an \'etale morphism, then $X_\infty = X \times_Y Y_\infty$. Applied to the open inclusion $U \to X$, we have $U_\infty= U \times_X X_\infty$. Applied to the \'etale map $U \to \mathbb{A}^n$ we have $U_\infty=U \times_{\mathbb{A}^n} \mathbb{A}^n_\infty$. Hence we have 
$$ \pi^{-1}(U) = U \times_X X_\infty= U_\infty= U \times_{\mathbb{A}^n} \mathbb{A}^n_\infty.$$ Hence $$\pi^{-1}(p)=\Spec \kappa(p) \times_U \pi^{-1}(U) = \Spec \kappa(p) \times_{\mathbb{A}^n} (\mathbb{A}^n)_\infty=(\mathbb{A}_{\kappa(p)}^n)_\infty.$$
\end{proof}

We resume considering Proposition \ref{base-case-cyl2}, where now it is sufficient to assume $X=\mathbb{A}^n=\Spec \CC[x_1, \ldots, x_n]$, and the $\CC$-valued point $p_0$ corresponds to the maximal ideal $(x_1, \ldots, x_n)$.
We write $(\mathbb{A}^n)_\infty=(\Spec \CC[x_1, \ldots, x_n])_\infty=\Spec \CC[\{x_{i,j}\}_{1 \leq i \leq n,\ j \geq 0}]$, where the last equality comes from parametrizing arcs on $\Spec \CC[x_1, \ldots, x_n]$ by $x_i \to \sum_{j \geq 0} x_{i,j}t^j$ for $1\leq i\leq n$. %In coordinates, the morphism $(\mathbb{A}^n)_\infty \to \mathbb{A}^n$ sending $\gamma$ to $\gamma(o)$ is given, on the level of rings, by $\CC[X_1, \ldots, X_n] \to \CC[\{x_{i,j}\}_{1 \leq i \leq n,\ j \geq 0}]$ sending $X_i \to x_{i,0}$.
Note that $\pi:X_\infty \to X$ (defined in Proposition \ref{reduction}) maps $C$ to $p_0$. Hence $$C \subseteq \pi^{-1}(p_0)=(\mathbb{A}^n)_\infty=\Spec S,$$ where 
\begin{equation}\label{S}
S=\CC[\{x_{i,j}\}_{1 \leq i \leq n,\ j \geq 1}] 
\end{equation} 

\begin{definition}\label{defn-of-f}
For $2\leq i\leq n$ and $q \geq 1$, let $f_{i,q}(X_1, \dots, X_q)$ be the polynomial that is the coefficient of $t^q$ in $$\sum _{j=1}^{q}c_{i,j}(X_1t + X_2t^2 + \cdots )^j.$$ (Recall that the $c_{i,j}$ were defined in Notation \ref{mess}).
\end{definition}

\begin{definition}\label{Iq}
For each positive integer $q$, let $I_q$ be the ideal of $S$ generated by
\begin{enumerate}
%\item $x_{i,0}$ for $1 \leq i \leq n$ and
\item $x_{i,j} - f_{i,j}(x_{1,1}, \ldots, x_{1,j})$ for $2 \leq i \leq n$ and $1\leq j \leq q-1$.
\end{enumerate} 
Note that $I_q$ is a prime ideal of $S$, since $ S/I_q = \CC[\{x_{1,j}\}_{j \geq 1},\{x_{i,j}\}_{2 \leq i \leq n, q \leq j}]$.
\end{definition}

\begin{notation}
If $J$ is an ideal of $S$, we denote by $V(J)$ the closed subscheme of $\Spec S$ defined by the ideal $J$.
\end{notation}

\begin{definition}\label{I}
Let $I$ be the ideal of $S$ defined by $I = \bigcup_{q \geq 1}I_q$. Since $I$ is the ideal of $S$ generated by $x_{i,j} - f_{i,j}(x_{1,1}, \ldots, x_{1,j})$ for $2 \leq i \leq n$ and $1\leq j$, we have $S/I = \CC[\{x_{1,j}\}_{1 \leq j}]$. In particular, $I$ is a prime ideal of $S$.
\end{definition}

\begin{lemma}\label{ig-lemma1}
For each positive integer $q$, the ideal of $\overline{\mu_{q\infty}(\Cont^{\geq 1}(E_q))}$ in $S$ is $I_q$. (Note: $I_q$ is defined in Definition \ref{Iq}.)
\end{lemma}

\begin{proof}
Note that $\overline{\mu_{q\infty}(\Cont^{\geq 1}(E_q))}$ is irreducible (e.g. \cite[p.9]{mustata}). Since $I_q$ is a prime ideal, we need to show $$\overline{\mu_{q\infty}(\Cont^{\geq 1}(E_q))} = V(I_q).$$

First we show $\overline{\mu_{q  \infty}(\Cont^{\geq 1}(E_q))} \subseteq V(I_q)$ by showing that the generic point of $\overline{\mu_{q\infty}(\Cont^{\geq 1}(E_q))}$ lies in $V(I_q)$. Suppose $\beta ' : \Spec K[[t]] \to X_q$ is the generic point of $\Cont^{\geq 1}(E_q)$. To be precise, $\beta '$ is the canonical arc (described in Remark \ref{point}) associated to the generic point of $\Cont^{\geq 1}(E_q)$. Also, $K$ is the residue field at the generic point of $\Cont^{\geq 1}(E_q)$. By Lemma \ref{lemma2} part \ref{lem2-ii}, the pushdown of $\beta '$ to $X_{q-1}$ is an arc $\beta:\Spec K[[t]] \to X_{q-1}$ that is the generic point of  $\Cont^{\geq 1}(p_{q-1})$. By the description of local coordinates at $p_{q-1}$ given in Lemma \ref{local-coords}, the arc $\beta$ corresponds (by Lemma \ref{lemma2}) to a map $x_1 \to x_{1,1}t + x_{1,2}t^2 + \cdots$ and $\frac{x_i-c_{i,1}x_1-c_{i,2}x^2_1  \cdots - c_{i,q-1}x^{q-1}_1}{x^{q-1}_1} \to a_{i,1}t + a_{i,2}t^2 + \cdots$ for $2 \leq i \leq n$ and some $a_{i, j} \in K$. The pushdown of $\beta$ to $X$ is the arc given by $x_1 \to x_{1,1}t + x_{1,2}t^2 + \cdots$ and 
$x_i \to \sum_{j=1}^{j=q-1}c_{i,j}(x_{1,1}t + x_{1,2}t^2 + \cdots)^j + r(t)$ where $r(t) \in (t^q) \subseteq K[[t]]$. In particular, the pushdown of $\beta '$ to $X$ corresponds to a prime ideal in $S$ containing the ideal $I_q$ of $S$ generated by  $x_{i,j}-f_{i,j}(x_{1,1}, \ldots, x_{1,j})$ for $1 \leq j \leq q-1$ and $2\leq i\leq n$. That is, the generic point of $\mu_{q  \infty}(\Cont^{\geq 1}(E_q))$  lies in $V(I_q)$. Hence $\overline{\mu_{q  \infty}(\Cont^{\geq 1}(E_q))} \subseteq V(I_q)$.

Conversely, we show that $\overline{\mu_{q\infty}(\Cont^{\geq 1}(E_q))} \supseteq V(I_q)$. The generators of $I_q$ listed in Definition \ref{Iq} show that the coordinate ring of $V(I_q)$ is $ S/I_q = \CC[\{x_{1,j}\}_{j \geq 1},\{x_{i,j}\}_{2 \leq i \leq n, q \leq j}]$. Let $\beta: \Spec K[[t]] \to X$ be the arc corresponding (see Remark \ref{point}) to the generic point of $V(I_q)$, where $K= \CC(\{x_{1,j}\}_{j \geq 1},\{x_{i,j}\}_{2 \leq i \leq n, q \leq j})$.  We have $\beta^*(x_1) = x_{1,1}t + x_{1,2}t^2 + \dots$. Since $I_q$ contains $x_{i,j}-f_{i,j}(x_{1, 1}, \ldots, x_{1, j})$ for $1 \leq j \leq q-1$ and $2 \leq i \leq n$, we have that $\beta^*(x_i) = \displaystyle\sum_{j \geq 1}^{q-1}f_{i,j}(x_{1,1}, \ldots, x_{1,j})t^j + t^qr_i(t)$ for some $r_i(t) \in K[[t]]$ and for each $2 \leq i\leq n$. Hence $$\beta^*(x_i)=\sum_{j \geq 1}^{q-1} c_{i,j}(\beta^*(x_1))^j+t^qs_i(t)$$ for some $s_i(t) \in K[[t]]$, by Definition \ref{defn-of-f}.

Therefore $$\ord_\beta(x_i-c_{i,1}x_1-c_{i,2}x^2_1  \cdots - c_{i,q-1}x^{q-1}_1) \geq q=\ord_\beta(x^{q-1}_1) + 1,$$ where the last equality follows from the fact $\ord_\beta(x_1)=1$ as $x_{1,1} \neq 0 \in K$. In particular, the unique lift of $\beta$ to an arc on $X_{q-1}$ has center $p_{q-1}$, by Lemma \ref{local-coords}. Hence $\beta \in \mu_{q-1 \infty}(\Cont^{\geq 1}(p_{q-1})) = \mu_{q\infty}(\Cont^{\geq 1}(E_q)).$ Hence $V(I_q)=\overline{\{\beta\}}\subseteq\overline{\mu_{q\infty}(\Cont^{\geq 1}(E_q))}$.
\end{proof}

\begin{lemma}\label{i-lemma1}
The ideal of $C$ in $S$ is $I$. (Note: $C$ is defined in Proposition \ref{base-case-cyl2}, $S$ is defined in Equation \ref{S}, and $I$ is defined in Definition \ref{I}.)
\end{lemma}

\begin{proof}
Since $I$ is a prime ideal, we need to show $C = V(I)$. We have  $$\bigcap_{q \geq 1} V(I_q) = V(\bigcup_{q \geq 1} I_q) = V(I)$$ and $$ C = \bigcap_{q \geq 1} \mu_{q\infty}(\Cont^{\geq 1}(E_q)) \subseteq \bigcap_{q \geq 1} V(I_q)$$ by Lemma \ref{ig-lemma1}. It remains to show $\bigcap_{q \geq 1} \mu_{q\infty}(\Cont^{\geq 1}(E_q)) \supseteq \bigcap_{q \geq 1} V(I_q)$.

Let $\beta: \Spec K[[t]] \to X$ be an arc corresponding to a point in $\bigcap_{q \geq 1} V(I_q)$. We may assume $\beta$ is not the trivial arc, since the trivial arc lies in $\bigcap_{q \geq 1}\mu_{q\infty}(\Cont^{\geq 1}(E_q))$. Say $\beta^*(x_1) = \sum_{j \geq 1} a_{1,j}t^j$, where $a_{1,j} \in K$. Since $I_{q}$ contains $x_{i,j}-f_{i,j}(x_{1, 1}, \ldots, x_{1, j})$ for $1 \leq j \leq q-1$ and $2 \leq i \leq n$, we have that $\beta^*(x_i) = \displaystyle\sum_{j \geq 1}^{\infty}f_{i,j}(a_{1,1}, \ldots, a_{1,j})t^j$ for each $2 \leq i\leq n$. Hence $\beta^*(x_i)=\sum_{j \geq 1}^{\infty} c_{i,j}(\beta^*(x_1))^j$, by Definition \ref{defn-of-f}.
Hence $$\ord_\beta({x_i-c_{i,1}x_1-c_{i,2}x^2_1  \cdots - c_{i,q-1}x^{q-1}_1})=\ord_\beta(\sum_{j\geq q}c_{i,j}x^j_1) = \ord_\beta{x^{q}_1} \geq \ord_\beta(x^{q-1}_1) + 1.$$ In particular, the unique lift of $\beta$ to an arc on $X_{q-1}$ has center $p_{q-1}$, by Lemma \ref{local-coords}. Hence $$\beta \in \mu_{q-1 \infty}(\Cont^{\geq 1}(p_{q-1})) = \mu_{q\infty}(\Cont^{\geq 1}(E_q)).$$ Hence $\bigcap_{q \geq 1}V(I_{q})\subseteq\bigcap_{q \geq 1}\mu_{q\infty}(\Cont^{\geq 1}(E_q))$.
\end{proof}

\begin{lemma}\label{descr-of-a-q}
For a positive integer $q$, let $\mathfrak{a}_q = \{f \in \widehat{\mathcal{O}}_{X, p_0} \mid v(f) \geq q\}$. 
Set $z_i=x_i-\sum_{j=1}^{q-1}c_{i,j}x^j_1$ for $2 \leq i \leq n$.
Then $\mathfrak{a}_q$ is generated (as an ideal in $\widehat{\mathcal{O}}_{X, p_0}$) by $x^q_1, z_2, \ldots, z_n$. 
\end{lemma}

\begin{proof}
By Equation \ref{calc-valn}, we have $v(x^q_1), v(z_i) \geq q$ for $2 \leq i \leq n$. Suppose $f \in \mathfrak{a}_q$. Since $\CC[[x_1, \ldots, x_n]]/(z_2, \ldots, z_n) \simeq \CC[[x_1]]$, we can write $f = \sum_{i \geq 2}^{i=n}h_iz_i + g(x_1)$, where $h_i \in \CC[[x_1, \ldots, x_n]]$ and $g(x_1) \in \CC[[x_1]]$. Then since $v(f)\geq q$, and $v(z_i)\geq q$, we must have $v(g)\geq q$. By Equation \ref{calc-valn}, we conclude $x^q_1$ divides $g(x_1)$ in $\CC[[x_1]]$. Hence $f$ is in the ideal generated by $x^q_1, z_2, \ldots, z_n$.
\end{proof}

\begin{lemma}\label{aq-iq}
For every positive integer $q$, the ideal of $\Cont^{\geq q}(\mathfrak{a}_q)$ in $S$ is $I_q$.
\end{lemma}

\begin{proof}
First we show $\Cont^{\geq q}(\mathfrak{a}_q) \subseteq V(I_q)$. Suppose $\beta:\Spec K[[t]] \to X$ is an arc corresponding (via Remark \ref{point}) to a generic point of $\Cont^{\geq q}(\mathfrak{a}_q)$. Write $\beta^*(x_i) = \x_{i,1}t + \x_{i,2}t^2 + \cdots$ for $1 \leq i\leq n$, where $\x_{i, j} \in K$ denotes the image in $K$ of $x_{i,j} \in S$. Since $\mathfrak{a}_q$ is generated by $x^q_1, z_2, \ldots, z_n$ (Lemma \ref{descr-of-a-q}) (recall that $z_i=x_i-\sum_{j=1}^{q-1}c_{i,j}x^j_1$ for $2 \leq i \leq n$), we have 
\begin{equation}\label{big}
\x_{i,1}t + \x_{i,2}t^2 + \cdots - \sum _{j=1}^{j=q-1}c_{i,j}(\x_{1,1}t + \x_{1,2}t^2 + \cdots )^j \in (t^q).
\end{equation}
The coefficient of $t^j$ in Equation \ref{big} is $\x_{i,j} - f_{i,j}(\x_{1,1}, \ldots, \x_{1,j})$. Hence $\beta$ corresponds to a prime ideal of $S$ containing the ideal $I_q$ of $S$ generated by  $x_{i,j} - f_{i,j}(x_{1,1}, \ldots, x_{1,j})$ for $2 \leq i \leq n$ and $1 \leq j \leq q-1$. Thus $\Cont^{\geq q}(\mathfrak{a}_q) \subseteq V(I_q)$.

Conversely, suppose $\beta:\Spec K[[t]] \to X$ corresponds (via Remark \ref{point}) to the generic point of 
$V(I)$. The coordinate ring of $V(I_q)$ is $ S/I_q = \CC[\{x_{1,j}\}_{j \geq 1},\{x_{i,j}\}_{2 \leq i \leq n, q \leq j}]$ (Definition \ref{Iq}). Hence $K$, the residue field at the generic point of $V(I_q)$, equals $ K = \CC(\{x_{1,j}\}_{j \geq 1},\{x_{i,j}\}_{2 \leq i \leq n, q \leq j})$. We have $\beta^*(x_1)= x_{1,1}t + x_{1,2}t^2 + \cdots \in K[[t]]$. Since $I_q$ contains $x_{i,j}-f_{i,j}(x_{1, 1}, \ldots, x_{1, j})$ for $1 \leq j \leq q-1$ and $2 \leq i \leq n$, we have that $\beta^*(x_i) = \displaystyle\sum_{j \geq 1}^{q-1}f_{i,j}(x_{1,1}, \ldots, x_{1,j})t^j + t^qr_i(t)$ for some $r_i(t) \in K[[t]]$ and for each $2 \leq i\leq n$. Since $\sum_{j \geq 1}c_{i,j}(x_{1,1}t + x_{1,2}t^2 + \cdots)^j = \sum_{j \geq 1}f_{i,j}(x_{1,1}, \dots, x_{1,j})t^j$ for $2 \leq i \leq n$ (Notation \ref{mess}),  we have that $\beta^*$ maps $x_i-c_{i,1}x_1-c_{i,2}x^2_1  \cdots - c_{i,q-1}x^{q-1}_1$ into the ideal $(t^q) \subseteq K[[t]]$. Hence by Lemma \ref{descr-of-a-q}, we have $\beta \in \Cont^{\geq q}(\mathfrak{a}_q)$. So $V(I_q) = \overline{\{\beta\}} \subseteq \Cont^{\geq q}(\mathfrak{a}_q)$.

\end{proof}

\begin{lemma}\label{i-lemma2}
The ideal of  $\bigcap_{q \geq 1}\Cont^{\geq q}(\mathfrak{a}_q)$ in $S$ is $I$. (Note: $S$ is defined in Equation \ref{S}, and $I$ is defined in Definition \ref{I}, and $\mathfrak{a}_q$ is defined in Proposition \ref{base-case-cyl2} (2).)
\end{lemma} 

\begin{proof}
Since $I$ is a prime ideal, it is enough to show $\bigcap_{q \geq 1}\Cont^{\geq q}(\mathfrak{a}_q) = V(I)$. 
By Lemma \ref{aq-iq}, we have $$\bigcap_{q \geq 1}\Cont^{\geq q}(\mathfrak{a}_q)=\bigcap_{q \geq 1} V(I_q) = V(\bigcup_{q \geq 1} I_q) = V(I).$$
\end{proof}

We now finish the proof of Proposition \ref{base-case-cyl2}.

\begin{proof}[Proof of Proposition \ref{base-case-cyl2}]
Since $S/I \simeq \CC[\{x_{1,j}\}_{j \geq 1}]$ is a domain, the ideal $I$ is a prime ideal. By Lemma \ref{i-lemma1}, the ideal of $C$ is $I$. Hence $C$ is irreducible. We  have $C=\bigcap_{q}\Cont^{\geq q}(\mathfrak{a}_q)$ because by Lemmas \ref{i-lemma1} and \ref{i-lemma2}, their ideals are the same. 

It remains to show $\val_C = v$. Let $\gamma:\Spec \CC[[t]] \to X$ be the arc centered at $p_0$ with $\gamma^*(x_1)= t$ and $\gamma^*(x_i) = P_i(t)$ for $2 \leq i \leq n$. Then $\gamma \in C$ since the ideal in $S$ corresponding to $\gamma$, namely the ideal generated by $x_{1,0}$, $x_{1,1}-1$, $x_{1,m}$, $x_{i,0}$, and $x_{i,j}-c_{i,j}$ for $m \geq 2$, $2 \leq i \leq n$, and $j \geq 1$ contains $I$. Hence for any $f \in \mathcal{O}_{X, p_0}$, we have $\val_C(f) \leq \ord_\gamma(f) = v(f)$.

For the reverse inequality, first suppose $f \in \mathcal{O}_{X, p_0}$ is such that $s:=v(f) < \infty$. Let $\gamma \in C$ be such that $\val_C(f)=\ord_\gamma(f)$.  Since $f \in \mathfrak{a}_s$ and $\gamma \in \Cont^{\geq s}(\mathfrak{a}_s)$, we have $\ord_\gamma(f) \geq s$, i.e. $\val_C(f) \geq v(f)$. 

Next suppose $v(f) = \infty$. Set $\phi_i = x_i-P_i(x_1)$ for $2 \leq i\leq n$. Since $$\CC[[x_1, \ldots, x_n]]/(\phi_2, \ldots, \phi_n) \simeq \CC[[x_1]],$$ we can write $f=\displaystyle \sum_{i=2}^{n}\phi_i{h_i} + g(x_1)$ for  $h_i \in \CC[[x_1, \ldots, x_n]]$ and $g \in \CC[[x_1]]$. Since $v(f)=\infty$, we have $g=0$ by Equation \ref{calc-valn}. Let $\gamma \in C$, and % and let $\gamma^*:\widehat{\mathcal{O}}_{X, p_0} \simeq \CC[[x_1, \ldots, x_n]] \to K[[t]]$ be the induced ring homomorphism extending the usual ring homomorphism $\mathcal{O}_{X, p_0} \to \CC[[t]]$. This extension to the completion exists because $\CC[[t]]$ is complete and $\gamma^*$ maps each $x_i$ into $(t)\CC[[t]]$, since $x_{i,0} \in I$ for $1 \leq i\leq n$. 
write $\gamma^*(x_1) = \sum_{j \geq 1} a_jt^j$. Since $x_{i,j} - f_{i,j}(x_{1,1}, \ldots, x_{1,j}) \in I$ for $2 \leq i \leq n$ and $j \geq 1$, we have $\gamma^*(x_i)= \sum_{j \geq 1}f_{i,j}(a_1, \ldots, a_j)t^j = \sum_{j\geq 1} c_{i,j}(a_1t + a_2t^2 + \ldots)^j = p_i(\gamma^*(x_1)) = \gamma^*(p_i(x_1))$. Hence $\gamma^*(\phi_i)=0$, and so $\gamma^*(f)=\gamma^*(\displaystyle \sum_{i=2}^{n}\phi_ih_i)=0$. So $\ord_\gamma(f) = \infty$. Since $\gamma \in C$ was arbitrary, we have $\val_C(f)=\infty$, as desired. 
\end{proof}

\subsection{General case}\label{general-section}

\begin{lemma}\label{key-lemma}
Let $X$ be a nonsingular variety of dimension $n$ ($n \geq 2$) over an algebraically closed field $\CC$ of characteristic zero. Let $\alpha:\Spec \CC[[t]] \to X$ be a normalized arc (Definition \ref{normalized}). Set $p_0=\alpha(o)$. Let $\alpha^*: \widehat{\mathcal{O}}_{X, p_0} \to \CC[[t]]$ be the $\CC$-algebra homomorphism induced by $\alpha$. Suppose $\gamma:\Spec \CC[[t]] \to X$ satisfies $\gamma(o)=p_0$ and $\ker(\alpha^*) \subseteq \ker(\gamma^*)$, where $\gamma^*:\widehat{\mathcal{O}}_{X,p_0} \to \CC[[t]]$ is the $\CC$-algebra homomorphism induced by $\gamma$. Assume $\gamma$ is not the trivial arc (Definition \ref{zero-arcs}).
Then 
\begin{enumerate}
\item  There exists a morphism $h: \Spec \CC[[t]] \to \Spec\CC[[t]]$ such that $\gamma = \alpha \circ h$, i.e. $\gamma$ is a reparametrization of $\alpha$. 
\item $h^*:\CC[[t]] \to \CC[[t]]$ is a local homomorphism.
\item Set $N=\ord_t(h)$. Then $\ord_\gamma=N\ord_\alpha$ on $\widehat{\mathcal{O}}_{X, p_0}$. (We use the convention that $\infty = N \cdot \infty$.)

\end{enumerate}
\end{lemma}

\begin{proof}{(Due to Mel Hochster.)}
We use Notation \ref{A-algebra}. Suppose $\gamma$ is not the trivial arc. By Lemma \ref{hochster-fix1}, $A_\gamma$ has dimension one, and so $\ker(\gamma^*)$ is a prime ideal of height $n-1$. The same is true for $\ker(\alpha^*)$, and so our assumption $\ker(\alpha^*) \subseteq \ker(\gamma^*)$ implies $\ker(\alpha^*)=\ker(\gamma^*)$. Hence $A_\alpha=A_\gamma$. By Lemma \ref{hochster-fix2}, the map $\alpha^*$ (resp. $\gamma^*$) induces an isomorphism $\overline{\alpha^*}: \tilde{A_\alpha} \to \CC[[\phi_\alpha]]$ (resp. $\overline{\gamma^*}: \tilde{A_\gamma} \to \CC[[\phi_\gamma]]$) for some $\phi_\alpha \in \CC[[t]]$ (resp. $\phi_\gamma \in \CC[[t]])$. Since $\alpha$ is normalized, we have $\ord_t(\phi_\alpha)=1$ by Proposition \ref{hochster-genius}. 

I claim that the inclusion $\CC[[\phi_\alpha]] \subseteq \CC[[t]]$ is actually an equality. It suffices to find $a_j \in \CC$ such that $t = \sum_{j \geq 1}a_j(\phi_\alpha)^j$. Suppose $\phi_\alpha = \sum_{j \geq 1} b_jt^j$, where $b_j \in \CC$ and $b_1 \neq 0$. We proceed to define $a_j$ by induction on $j$. Set $a_1 = {b_1}^{-1}$. Suppose $a_1, \ldots, a_{d-1}$ have been specified, for some $d \geq 2$. The coefficient of $t^{d}$ in $\sum_{j \geq 1}a_j(\phi_\alpha)^j$ is $a_db_1^d + Q_d(a_1, \ldots, a_{d-1}, b_1, \ldots, b_d)$ for some polynomial $Q_d$. We require this coefficient to be $0$ since we want $t = \sum_{j \geq 1}a_j(\phi_\alpha)^j$. We can solve the equation $$a_db_1^d + Q_d(a_1, \ldots, a_{d-1}, b_1, \ldots, b_d)=0$$ for $a_d$ since $b_1 \neq 0$. This completes the induction, and we have $t = \sum_{j \geq 1}a_j(\phi_\alpha)^j$.

Let $h:\Spec \CC[[t]] \to \Spec \CC[[t]]$ be induced by the $\CC$-algebra homomorphism $h^*:\CC[[t]] \to \CC[[t]]$ defined by the composition $$\CC[[t]] = \CC[[\phi_\alpha]] \xrightarrow{(\overline{\alpha^*})^{-1}}\tilde{A}_\alpha = \tilde{A}_\gamma \xrightarrow{\overline{\gamma^*}} \CC[[\phi_\gamma]] \subseteq \CC[[t]].$$ 
The last inclusion is an inclusion of local $\CC$-algebras and all other maps are isomorphisms. Hence $h^*$ is a local homomorphism. For $f \in \widehat{\mathcal{O}}_{X,p_0}$, we have $\gamma^*(f)=\overline{\gamma^*}(f) =  h^* \circ \overline{\alpha^*}(f)= h^* \circ \alpha^*(f)$, and hence $\gamma = \alpha \circ h$.
If $\ord_t(h)=N$ and $a=\ord_\alpha(f)$, then the order of $t$ in $\gamma^*(f)=h^* \circ \alpha^*(f)$ is $Na$, i.e. $\ord_\gamma(f)=N\ord_\alpha(f)$.
\end{proof}

\begin{notation}\label{little-k}
We denote by $(X_\infty)_0$ the subset of points of $X_\infty$ with residue field equal to $\CC$. If $D \subseteq X_\infty$, then we set $D_0 = D \cap (X_\infty)_0$.
\end{notation}

Here is the main theorem of this paper. 

\begin{theorem}\label{the-thm}
Let $X$ be a nonsingular variety of dimension $n$ ($n \geq 2$) over a field $\CC$. Let $\alpha:\Spec \CC[[t]] \to X$ be a normalized arc (Definition \ref{normalized}). Set $p_0=\alpha(o)$ and $v = \ord_\alpha$. Let $E_i$ and $p_i$ be the sequence of divisors and centers, respectively, of $v$  (described in Definition \ref{divisor-limit}). Let $\mu_q:X_q \to X$ be the composition of the first $q$ blowups of centers of $v$. Let 
\begin{equation}\label{C}
C = \bigcap_{q>0} \mu_{q \infty}(\Cont^{\geq 1}(E_q))\subseteq X_\infty.
\end{equation}

Let $\mathfrak{a}_q = \{f \in \widehat{\mathcal{O}}_{X, p_0} \mid v(f) \geq q\}$. Let $$C''=\bigcap_{q \geq 1} \Cont^{\geq q}(\mathfrak{a}_q) \subseteq X_\infty.$$

Set $C(v)=\overline{\{\gamma \in X_\infty \mid \ord_\gamma = v, \ \gamma(o)=p\}} \subseteq X_\infty.$

Let $Y=\overline{\alpha(\eta)}$ where $\eta$ is the generic point of $\Spec \CC[[t]]$. Let $\pi:X_\infty \to X$ be the canonical morphism sending an arc $\gamma \in X_\infty$ to its center $\gamma(o) \in X$.

Let $R=\{a \circ h \in X_\infty \mid h:\Spec \CC[[t]] \to \Spec \CC[[t]]\}$, where $h$ is a morphism of $\CC$-schemes. 

Then 
\begin{enumerate}

\item \label{irreducible} $C$ is an irreducible subset of $X_\infty$ and $\val_C = v$.

\item \label{ideals-thm} Assume $\CC$ is algebraically closed and has characteristic zero. The following closed subsets of $(X_\infty)_0$ are equal (we use Notation \ref{little-k}):
$$C(v)_0=C_0 = {C''}_0=(\pi_X^{-1}(\alpha(o)) \cap Y_\infty)_0 =R.$$

%\item \label{maximal} If $E \subseteq X_\infty$ is an irreducible subset such that $\pi(E)=p_0$ and $\val_{E}=v$, then $E \subseteq C$. 

\end{enumerate}
\end{theorem}

\

\

\begin{proof}[Proof of Theorem \ref{the-thm}.]
(Part \ref{irreducible}) Let $r$ be a nonnegative integer such that the lift of $\alpha$ to $X_r$ is a nonsingular arc. For $q >r$, let $\mu_{q, r}:X_q \to X_r$ be the composition of the blowups along the centers of $v$, starting at $X_{r+1} \to X_r$ and ending at the blowup $X_q \to X_{q-1}$. Let $$C'=\displaystyle \bigcap_{q > r}\mu_{q, r \infty}(\Cont^{\geq 1}(E_q)) \subseteq (X_r)_\infty.$$ Note that $$C = \mu_{r \infty}(C') \subseteq X_\infty.$$

By Proposition \ref{base-case-cyl2}, $C'$ is irreducible. Hence $C$ is irreducible. Since the generic point of $C'$ maps to the generic point of $C$, we have that $\val_{C'}=\val_C$, i.e. $\val_{C'}(\mu^*_r(f))=\val_C(f)$ for $f \in \mathcal{O}_{X,p_0}$. Since $v = \val_{C'}$ by Proposition \ref{base-case-cyl2}, we conclude $v = \val_C$.

(Part \ref{ideals-thm}) We show $C(v)_0 \subseteq {C''}_0 \subseteq C_0 \subseteq C(v)_0$. Separately we will establish $C''=\pi^{-1}(p_0)\cap Y_\infty$.

First we check $C(v) \subseteq C''$. If $\gamma \in X_\infty$ is such that $\gamma(o)=p$ and $\ord_\gamma = v$, then $\gamma \in \Cont^{\geq q}(\mathfrak{a}_q)$ for every $q \geq 1$, and so $\gamma \subseteq C''$. Since $C''$ is closed, we have $C(v) \subseteq C''.$

Now we show ${C''}_0 \subseteq C_0$. Let $\gamma \in {C''}_0$, and assume without loss of generality that $\gamma$ is not the trivial arc. We claim that $\ker(\alpha^*) \subseteq \ker(\gamma^*)$. Let $f \in \ker(\alpha^*)$. Then $v(f)=\infty$, and so $f \in \mathfrak{a}_q$ for every $q \in \mathbb{Z}_{\geq 0}$. Hence $\ord_\gamma(f) \geq q$ for all $q \in \mathbb{Z}_{\geq 0}$. Therefore $\ord_\gamma(f)=\infty$, so $f \in \ker(\gamma^*)$. By Lemma \ref{key-lemma} there exists $h: \Spec \CC[[t]] \to \CC[[t]]$ such that $\gamma = \alpha \circ h$. It follows that $\gamma$ has the same sequence of centers as $\alpha$. Indeed, if $\gamma_q: \Spec \CC[[t]] \to X_q$ is the unique lift of $\gamma$ to an arc on $X_q$, then $\gamma_q \circ h$ is the unique lift of $\alpha$ to an arc on $X_q$. Since $h^*$ is a local homomorphism, we have that $h$ maps the closed point of $\Spec \CC[[t]]$ to the closed point of $\Spec \CC[[t]]$. Hence the center of $\gamma_q$ is the same as the center of $\gamma_g \circ h$. We conclude $\gamma \in C$. Note that this argument also shows ${C''}_0 \subseteq R$, and Lemma \ref{key-lemma} shows that ${C''}_0 \subseteq R$.

To see that $C \subseteq C(v)$, let $\beta$ be the generic point of $C$. Note that $\ord_\beta=v$ and $\pi(\beta)=p_0$, and so $\beta \in C(v)$. Hence $C \subseteq C(v)$. 

Now we show ${C''}_0=(\pi^{-1}(p_0)\cap Y_\infty)_0$. 
Let $J$ be the kernel of the map $\alpha^*: \widehat{\mathcal{O}}_{X,p_0} \to \Spec \CC[[t]]$. If $f \in J$, then $\ord_\alpha=\infty$ and hence $f \in \mathfrak{a}_q$ for every $q \geq 1$. Let $\gamma \in {C''}_0$. Since $\mathfrak{a}_1$ is the maximal ideal of $\widehat{\mathcal{O}}_{X, p_0}$, we have $\gamma(o)=p_0$, i.e. $\gamma \in \pi^{-1}(p_0)$. Also, since $\ord_\gamma(f) \geq q$ for every $q \geq 1$, we have $\ord_\gamma(f) = \infty$. Hence $\gamma \in (\pi^{-1}(p_0)\cap Y_\infty)_0$.

For the reverse inclusion ${C''}_0 \supseteq (\pi^{-1}(p_0)\cap Y_\infty)_0$, let $\gamma \in (\pi^{-1}(p_0)\cap Y_\infty)_0$. Then $J \subseteq \ker(\gamma^*)$, and hence by Lemma \ref{key-lemma} we have that either $\gamma$ is the trivial arc or $\ord_\gamma=N\ord_\alpha$ for some positive integer $N$. In both cases we have $\gamma \in {C''}_0$.

\end{proof}

\begin{remark}
If $X$ is a surface and if $v$ is a divisorial valuation, then the set $$C = \bigcap_{q>0} {\mu_{q \infty}(\Cont^{\geq 1}(E_q))}$$ equals the cylinder associated to $v$ in \cite[Example 2.5]{mustata}, namely ${\mu_{r \infty}(\Cont^{\geq 1}(E_r))}$, where $r$ is such that $p_r$ is a divisor. 
\end{remark}

\begin{proof}
 If $r$ is such that $p_r \in X_r$ (Definition \ref{divisor-limit}) is a divisor, then $C={\mu_{r \infty}(\Cont^{\geq 1}(E_r))}$ since ${\mu_{q \infty}(\Cont^{\geq 1}(E_q))} \supseteq {\mu_{q+1 \infty}(\Cont^{\geq 1}(E_{q+1}))}$, and for $q>r$ we have equality since the maps $\mu_{q, r}$ are isomorphisms. Hence $C={\mu_{r \infty}(\Cont^{\geq 1}(E_r))}$, which is the set in \cite[Example 2.5]{mustata}.
\end{proof}


\begin{thebibliography}{10}

\bibitem[DGN]{delgado} F. Delgado, C. Galindo, A. Nunez. \textit{Saturation for valuations on two-dimensional regular local rings}. Math. Z. 234, 519-550 (2000).

\bibitem[ELM]{mustata} L. Ein, R. Lazarsfeld, M. Musta\c{t}\v{a}. \textit{Contact Loci in Arc Spaces}. Compositio Math. 140 (2004) 1229-1244.


\bibitem[EM]{jet-bible} L. Ein, M. Musta\c{t}\v{a}. \textit{Jet Schemes and Singularities}. arXiv:math.AG/0612862v1 29 Dec 2006.

\bibitem[Eisenbud]{E} D. Eisenbud. \textit{Commutative Algebra with a View Towards Algebraic Geometry.} Springer-Verlag, 1995.

\bibitem[FJ]{mattias} C. Favre, M. Jonsson. \textit{The Valuative Tree.} Lecture Notes in Mathematics, 1853. Springer-Verlag, Berlin, 2004. 


\bibitem[Gal]{galindo} C. Galindo. \textit{Plane Valuations and their Completions}. Comm. in Algebra, 23(6), pp. 2107-2123, (1995).

\bibitem[H]{hartshorne} R. Hartshorne. \textit{Algebraic Geometry}. Springer-Verlag, 1977.

\bibitem[Ishii]{ishii} S. Ishii. \textit{Arcs, Valuations, and the Nash map}. J. reine angew Math 588 (2005), 71-92.

\bibitem[Ishii2]{ishii3} S. Ishii. \textit{Maximal Divisorial Sets in Arc Spaces}. arXiv:math.AG/0612185 v1 7 Dec 2006. To appear in Proceeding Algebraic Geometry in East Asia--Hanoi 2005.

\bibitem[Ishii3]{ishii2} S. Ishii. \textit{Jet schemes, arc spaces, and the Nash problem}. arXiv:math.AG/0704.3327v1 25 Apr 2007. To appear in the Mathematical Reports of the Academy of Science of the Royal Society of Canada (Canadian Comptes Rendus).


\bibitem[Matsumura80]{M} H. Matsumura. \textit{Commutative Algebra.} Second Edition. Benjamin/Cummins Publishing Co., 1980.

\bibitem[Milne]{milne} J. Milne. \textit{\'Etale Cohomology}. Princeton University Press, 1980.


%\bibitem[Serre]{serre} J. Serre. \textit{Local Fields.} Springer-Verlag. 1979. 

\bibitem[Spiv]{spivak} M. Spivakovsky. \textit{Valuations in Function Fields of Surfaces}. American Journal of Mathematics, Vol. 112, No. 1. (Feb., 1990), pp 107-156.

\bibitem[Vaq]{vaquie} M. Vaquie. \textit{Valuations}. Resolution of singularities (Obergurgl, 1997),  539--590, Progr. Math., 181, Birkhäuser, Basel, 2000. 

\end{thebibliography}
\end{document}